\documentclass[preprint,10pt]{elsarticle}

\pdfoutput=1

\usepackage{algorithm}
\usepackage{algpseudocode}
\usepackage{algcompatible}

\usepackage{graphicx}
\usepackage{amssymb,amsmath,bm,mathrsfs}
\usepackage{amsthm}
\usepackage{mathabx}
\usepackage{array}

\usepackage[final]{hyperref} 

\newtheorem{thm}{Theorem}
\newtheorem{lem}[thm]{Lemma}
\newtheorem{example}[thm]{Example}
\newtheorem{assumption}[thm]{Assumption}
\newtheorem{rmk}[thm]{Remark}
\newtheorem{definition}[thm]{Definition}
\newtheorem{corollary}[thm]{Corollary}
\newtheorem{proposition}[thm]{Proposition}
\newproof{pf}{Proof}
\numberwithin{thm}{section}

\newcommand {\rank}     {\mathop{\rm rank}\nolimits}

\newcommand {\range}  {\mathop{\rm range}\nolimits}
\newcommand {\vspan}  {\mathop{\rm span}\nolimits}

\newcommand {\kernel}   {\mathop{\rm kernel}\nolimits}

\newcommand{\const}{\mbox{\rm const.}}

\newcommand{\diag}{\mbox{\rm diag}}

\def\leq{\leqslant}

\def\r{\hro{R}}

\def\hro{\mathbb}
\def\N{\hro{N}}

\def\a{\alpha}

\def\vphi{\varphi}

\def\De{\Delta}

\def\Si{\Sigma}

\def\tA{\tilde{A}}
\def\hA{\hat{A}}

\def\tB{\tilde{B}}
\def\tD{\tilde{D}}
\def\hB{\hat{B}}
\def\hr{\hat{r}}
\def\hv{\hat{v}}
\def\tr{\tilde{r}}
\def\tv{\tilde{v}}

\def\cM{\mathcal M}
\def\tcM{\tilde{\mathcal M}}
\def\cN{\mathcal N}
\def\cX{\mathcal X}

\def\cU{\mathcal U}

\def\hD{\hat{D}}

\def\tM{\tilde{M}}

\def\tC{\tilde{C}}
\def\hC{\hat{C}}
\def\hD{\hat{D}}

\def\hG{\hat{G}}
\def\cG{{\cal G}}

\def\hg{\hat{g}}

\def\cW{{\cal W}}

\def\tf{\tilde{f}}
\def\hf{\hat{f}}

\def\chA{\widecheck{A}}
\def\chB{\widecheck{B}}
\def\chC{\widecheck{C}}
\def\chD{\widecheck{D}}
\def\chG{\widecheck{G}}
\def\chU{\widecheck{U}}

\def\lsim{\overset{\ell}{\sim}}

\newcommand{\Jbone}{	J_{\textrm{\tiny \textsc{B1}}}	}
\newcommand{\Jbtwo}{	J_{\textrm{\tiny \textsc{B2}}}	}
\newcommand{\Jcone}{	J_{\textrm{\tiny \textsc{C1}}}	}
\newcommand{\Jctwo}{	J_{\textrm{\tiny \textsc{C2}}}	}
\newcommand{\Jd}{	J_{\textrm{\tiny \textsc{D}}}	}

\def\be{\begin{equation}}
\def\ee{\end{equation}}         

\newcommand{\ben}{\begin{eqnarray}}
\newcommand{\een}{\end{eqnarray}}

\newcommand{\bea}{\begin{align}}
\newcommand{\eea}{\end{align}}

\newcommand{\bsen}{\begin{subeqnarray}}
\newcommand{\esen}{\end{subeqnarray}}

\newcommand{\bens}{\begin{eqnarray*}}
\newcommand{\eens}{\end{eqnarray*}}

\def\bc{\begin{cases}}
\def\ec{\end{cases}}

\newcommand{\bsq}{\begin{subequations}}
\newcommand{\esq}{\end{subequations}}

\newcommand{\m}[1]{
	\begin{bmatrix}
		#1 
	\end{bmatrix}
}

\renewcommand{\pm}[1]{
	\begin{matrix}
		#1 
	\end{matrix}
}

\newcommand{\sm}[1]{
	\left[
	\begin{smallmatrix}
		#1 
	\end{smallmatrix}
	\right]
}

\newcommand\undermat[2]{%
	\makebox[0pt][l]{$\smash{\underbrace{\phantom{%
					\begin{matrix}#2\end{matrix}}}_{\text{$#1$}}}$}#2}
\begin{document}

\begin{frontmatter}

\title{Index Reduction for Second Order Singular Systems \\of Difference Equations}

\author[add1]{Vu Hoang Linh\corref{cor1}} 
\ead{linhvh@vnu.edu.vn}

\author[add1]{Ha Phi}
\ead{haphi.hus@vnu.edu.vn}

\address[add1]{Faculty of Mathematics, Mechanics, and Informatics, Vietnam National University, 334, Nguyen Trai, Thanh Xuan, Hanoi, Vietnam.}

\cortext[cor1]{Corresponding author}

\begin{abstract}
This paper is devoted to the analysis of linear second order \emph{discrete-time descriptor systems} (or singular difference equations (SiDEs) with control). 
Following the algebraic approach proposed by Kunkel and Mehrmann for pencils of matrix valued functions, first we present a theoretical framework based on a procedure of reduction to analyze solvability of initial value problems for SiDEs, which is followed by the analysis of descriptor systems.
We also describe methods to analyze structural properties related to the solvability analysis of these systems. Namely, two numerical algorithms for reduction to the so-called \emph{strangeness-free forms} are presented. Two associated index notions are also introduced and discussed.  
This work extends and complements some recent results for high order continuous-time descriptor systems and first order discrete-time descriptor systems.
\end{abstract}
	


\end{frontmatter}

\section{Introduction}\label{intro}

In this paper we study second order discrete-time descriptor systems of the form
\begin{equation}\label{eq1.1}
A_n x(n+2) + B_{n} x(n+1) + C_{n} x(n) + D_{n} u(n)  = f(n)\ \mbox{ for all } n \geq n_0. 
\end{equation}
We will also discuss the initial value problem of the associated singular difference equation (SiDE) 
\begin{equation}\label{eq1.2}
A_n x(n+2) + B_{n} x(n+1) + C_{n} x(n) = f(n)\ \mbox{ for all } n \geq n_0,
\end{equation}
together with some given initial conditions
\begin{equation}\label{eq1.3}
x(n_0+1)=x_{1}, \ x(n_0)=x_{0}.
\end{equation}
Here the solution/state $x=\{ x(n) \}_{n\geq n_0}$, the inhomogeneity $f=\{ f(n) \}_{n\geq n_0}$, the input $u=\{ u(n) \}_{n\geq n_0}$, where $x(n) \in \r^d$, $f(n) \in \r^m$ and $u(n) \in \r^p$ for each $n \geq n_0$. Three matrix sequences $\{A_{n}\}_{n\geq n_0}$, $\{B_{n}\}_{n\geq n_0}$, $\{C_{n}\}_{n\geq n_0}$ take values in $\r^{m,d}$, and $\{D_{n}\}_{n\geq n_0}$ takes values in $\r^{m,p}$. 
We notice that all the results in this paper also can be carried over to the complex case and they can also be easily extended to systems of higher order. However, for sake of simplicity and because this is the most important case in practice, we restrict ourselves to the case of real and second order systems.

The SiDE \eqref{eq1.2}, on one hand, can be considered as the resulting equation obtained by finite difference or discretization of some continuous-time DAEs or constrained PDEs. On the other hand, there are also many models/applications in real-life, which lead to SiDEs, for example Leotief economic models, biological backward Leslie model, etc, see e.g. \cite{Aga00,Ela13,Kel01,Lue79}. 

While both DAEs and SiDEs of first order have been well-studied from both theoretical and numerical points of view,
the same maturity has not been reached for higher order systems. 
In the classical literature for regular difference equations, e.g. \cite{Aga00,Ela13,Kel01}, usually new variables are introduced to represent some chosen derivatives of the state variable $x$ such that a high order system can be reformulated as a first order one. Unfortunately for singular systems, this approach may induce some substantial disadvantages. As have been fully discussed in \cite{LosM08,MehS06} for continuous-time systems, these disadvantages include: (1st) increase the index of the singular system, and therefore the complexity of a numerical method to solve it; (2nd) increase the computational effort due to the bigger size of a new system; (3rd) affect the controllability/observability of the corresponding descriptor system since there exist situations where a new system is uncontrollable while the original one is. 
Therefore, the \emph{algebraic approach}, which treats the system directly without reformulating it, has been presented in \cite{LosM08,MehS06,Wun06,Wun08} in order to overcome the disadvantages mentioned above.
Nevertheless, even for second order SiDEs, this method has not yet been considered.

Another motivation of this work comes from recent research on the stability analysis of high order discrete-time systems with time-dependent coefficients \cite{LinNT16,MehT15}. In these works, systems are supposed to be given in either strangeness-free form or linear state-space form. This, however, is not always the case in applications, and hence, a reformulation procedure would be required.


%
%
Therefore, the main aim of this article is to set up a comparable framework for second order SiDEs and for discrete-time descriptor systems as well. It is worth marking that the algebraic method proposed in \cite{LosM08,MehS06} is applicable theoretically but not numerically due to two reasons: (1st) The condensed forms of the matrix coefficients are really big and complicated; (2nd) The system's transformations are not orthogonal, and hence, not numerically stable. In this work, we will modify this method to make it more concise and also computable in a stable way.

The outline of this paper is as follows. After giving some auxiliary results in Section \ref{pre}, in Sections \ref{Sec2} and \ref{Sec3} we consecutively introduce \emph{index reduction procedures} for SiDEs and for descriptor systems. 
A desired \emph{strangeness-free form} and a constructive algorithm to get it will be presented in Theorem \ref{thm2} and Algorithm \ref{Alg1} (Section \ref{Sec2}). 
A resulting system from this algorithm allows us to fully analyze structural properties such as existence and uniqueness of a solution, consistency and hidden constraints, etc. 
For descriptor systems, where feedback also takes part in the regularization/solution procedure, besides the strangeness-free form presented in Theorem \ref{thm4}, regularization via first order feedback is discussed in Theorem \ref{pro5.1}.  
In order to get stable numerical solutions of these systems, in Section \ref{Sec4} we study the \emph{difference array approach} in Algorithm \ref{Alg2} and Theorem \ref{thm6} aiming at bringing out the strangeness-free form of a given system. 
Finally, we finish with some conclusions.\\

\section{Preliminaries} \label{pre}
In the following example we demonstrate some difficulties that may arise in the analysis of second order SiDEs.
\begin{example}\label{Exa1}
\rm{	
	Consider the following second order descriptor system, motivated from Example 2, \cite{MehS06}.
	\be\label{eq1.4}
	\m{1 & 0 \\ 0 & 0} x(n\!+\!2) \!+\! \m{1 & 0 \\ 0 & 0} x(n\!+\!1) \!+\! \m{0 & 1 \\ 1 & 0} x(n) \!-\! \m {1 \\ 1} u(n) \!= \m{f_1(n) \\ f_2(n)}, \ n\geq n_0.  
	\ee
	Clearly, from the second equation $\m{1 & 0} x(n) = u(n) + f_2(n)$, we can shift the time $n$  forward to obtain 
	\[
	\m{1 & 0} x(n+1) = u(n+1) + f_2(n+1)  \ \mbox{ and } \ \m{1 & 0} x(n+2) = u(n+2) + f_2(n+2).
	\]
	Inserting these into the first equation of \eqref{eq1.4}, we find out the hidden constraint 
	\[
	f_2(n+2) + u(n+2) + f_2(n+1) + u(n+1) + \m{0 & 1} x(n) = f_1(n) \ .
	\]
	Consequently, we deduce the following system, which possess a unique solution
	\[
	\m{0 & 1 \\ 1 & 0} x(n) = \m{f_1(n) - f_2(n+2) - f_2(n+1) - u(n+2) - u(n+1)  \\ u(n) + f_2(n)}, \ n\geq n_0.
	\]
	Let $n=n_0$ in this new system, we obtain a constraint that $x(n_0)$ must obey.
	This example showed us some important facts. Firstly, one can use some shift operators and row-manipulation (Gaussian eliminations) to derive hidden constraints. Secondly, a solution only exists if initial conditions and an input fulfill certain consistency conditions. Finally, in this example the solution depends on the future input. This property is called \emph{non-causality} and cannot happen in the case of regular difference equations. 
}
\end{example}

For matrices $Q\in \r^{q,d}$, $P\in\r^{p,d}$, the pair $(Q,P)$ is said to
\emph{have no hidden redundancy} if
\[
\rm{rank} \left(\m{Q \\ P} \right) = \rm{rank} (Q) + \rm{rank}(P).
\]
Otherwise, $(Q,P)$ is said to \emph{have hidden redundancy}.
The geometrical meaning of this concept is that the intersection space $\vspan(P^T) \cap \vspan(Q^T)$ contains only the zero-vector $0$. Here, for any given matrix $M$, by $M^T$ we denote its transpose. We denote by $\vspan(P^T)$ (resp., $\vspan(Q^T)$) the real vector space spanned by the rows of $P$ (resp., rows of $Q$). 

%
%
\begin{lem}\label{lem1.3}(\cite{HaM12})
Consider $k+1$ full row rank matrices $R_0 \in \r^{r_0,d}$, $R_1 \in \r^{r_1,d}\dots, R_k \in \r^{r_k,d}$, and assume that for $j=k,\dots,1$ none of the matrix pairs \linebreak $\left(R_j, \m{R^T_{j-1} \ \hdots \ R^T_0}^T \right)$ has a hidden redundancy. Then $\m{R^T_k \ \hdots \ R^T_0}^T$ has full row rank.
\end{lem}
Lemma \ref{lem1.2} below will be very useful later for our analysis, in order to remove hidden redundancy in the coefficients of \eqref{eq1.2}.

\begin{lem}\label{lem1.2} Consider two matrix sequences $\{P_{n}\}_{n\geq n_0}$, $\{Q_{n}\}_{n\geq n_0}$ which take values in $\r^{p,d}$ and $\r^{q,d}$, respectively. Furthermore, assume that they satisfy the constant rank assumptions
	\[
	\rank \left(Q_n\right) = r_Q , \ \mbox{ and } \rank \left(\sm{ P_n \\ Q_n } \right) = r_{[P;Q]} \ \mbox{ for all } n\geq n_0 \ .
	\]
	Then there exists a matrix sequence $\Big\{ \sm{S_n & 0 \\ Z^{(1)}_{n} & Z^{(2)}_{n}} \Big\}_{n\geq n_0}$ in $\r^{p,p+q}$ such that the following conditions hold.
	\begin{enumerate}
		\item[i)] $S_n \in \r^{r_{[P;Q]}-r_Q, \ p}$, $Z^{(1)}_{n} \in \r^{p-r_{[P;Q]}+r_Q, \ p}$, $Z^{(2)}_{n} \in \r^{p-r_{[P;Q]}+r_Q, \ q}$,
		\item[ii)] $\sm{S_n \\ Z^{(1)}_{n}} \in \r^{p,p} $ is orthogonal, and $Z^{(1)}_{n} P_{n} + Z^{(2)}_{n} Q_{n} = 0$,
		\item[iii)] $S_n P_{n}$ has full row rank, and the pair $(S_n P_{n},Q_{n})$ has no hidden redundancy.
	\end{enumerate}
\end{lem}
\begin{pf} Since the proof is essentially the same as in the continuous-time case, we refer the interested readers to the proof of Lemma 2.7, \cite{HaMS14}.
\end{pf}

\begin{rmk}\label{rem1}
\rm{	
	i) In the special case, where $P_n$ has full row rank and the pair $(P_n,Q_n)$ has no hidden redundancy, we will adapt the notation of an empty matrix and take 
	$S_n=I_p$, $Z^{(1)}_{n} = [ \ ]_{0,p}$, $Z^{(2)}_{n}=[ \ ]_{0,q}$. \\
	ii) Furthermore, we notice, that whenever the smallest singular value of $Q_n$ and the largest one do not differ very much in size, then we can stably compute the matrix $Z^{(2)}_{n}$. Both matrices $Z^{(1)}_{n}$ and $Z^{(2)}_{n}$ will play the key role in our \emph{index reduction procedure} presented in the next section.
}
\end{rmk}

For any given matrix $M$, by $T_0(M)$ we denote an orthogonal matrix whose columns span the left null space of $M$. By $T_{\perp}(M)$ we denote an orthogonal matrix whose columns span the vector space $\range(M)$. 
From basic linear algebra, we have the following lemma.

\begin{lem}\label{lem1.4} The matrix $\m{T^T_{\perp}(M) \\ T^T_0(M)}$ is nonsingular, the matrix $T^T_{\perp}(M)  M$ has full row rank, and the following identity holds
	\[
	\m{T^T_{\perp}(M) \\ T^T_0(M)} M = \m{T^T_{\perp}(M) M \\ 0 }.
	\]
\end{lem}
\begin{pf}
	A simple proof can be found, for example, in \cite{GolV96}.
\end{pf}

\section{Strangeness-index of second order SiDEs}\label{Sec2}

In this section, we study the solvability analysis of the second order SiDE \eqref{eq1.2} and that of its corresponding IVP \eqref{eq1.2}--\eqref{eq1.3}. Many regularization procedures and their associated index notions have been proposed for first order systems, see the survey \cite{Meh13} and the references therein. Nevertheless, for high order systems, only the strangeness-index has been proposed in the continuous-time case in \cite{MehS06,Wun08}. 
Thus, it is our purpose to construct a comparable regularization and index concept for discrete-time system \eqref{eq1.2}. 

Let
\begin{equation*}
M_n \!:= \m{A_{n} & B_{n} & C_{n}}, \
X(n) \!:=\! \m{x(n+2) \\ x(n+1) \\ x(n)},
\end{equation*}
we call $\{M_n\}_{n\geq n_0}$ the \emph{behavior matrix sequence} of system  \eqref{eq1.2}. Thus,   \eqref{eq1.2} can be rewritten as
\begin{equation}\label{eq2.2}
M_n X(n) = f(n)\ \mbox{ for all } n\geq n_0.
\end{equation}
Clearly, by scaling  \eqref{eq1.2} with a pointwise nonsingular matrix sequence $\{P_n\}_{n\geq n_0}$ in $\r^{m,m}$, we obtain a new system
\begin{equation}\label{eq2.3}
\m{P_nA_{n} & P_nB_{n} & P_nC_{n}} X(n) = P_n f(n)\ \mbox{ for all } n\geq n_0,
\end{equation}
without changing the solution space. This motivates the following definition.
\begin{definition}\label{defstrleq}	
\rm{
Two behavior matrix sequences $\{M_n =\m{A_{n} & B_{n} & C_{n}} \}_{n\geq n_0}$ and $\{\tM_n =\m{\tA_{n} & \tB_{n} & \tC_{n}}\}_{n\geq n_0}$ are called \emph{(strongly) left equivalent} if there exists a pointwise nonsingular matrix sequence $\{P_n\}_{n\geq n_0}$ such that $\tM_n = P_n M_n$ for all $n\geq n_0$. We denote this equivalence by $\{M_n\}_{n\geq n_0} \lsim \{\tM_n\}_{n\geq n_0}$.	If this is the case, we also say that two SiDEs  \eqref{eq1.2}, \eqref{eq2.3} are left equivalent.
}
\end{definition}

\begin{lem}\label{lem2.1}
	Consider the behavior matrix sequence $\{M_n\}_{n\geq n_0}$ of system  \eqref{eq1.2}. Then for all $n\geq n_0$, we have that
	\begin{equation}\label{block-upper-M}
	\{M_n\}_{n\geq n_0} \  \lsim  \ \left\{ 
	\m{A_{n,1}& B_{n,1}    & C_{n,1}     \\
		0	& B_{n,2}    & C_{n,2}     \\
		0	&    0          & C_{n,3} \\  
		0     & 0            & 0} \right\}_{n\geq n_0}, \qquad
	\pm{r_{2,n} \\ r_{1,n} \\ r_{0,n} \\ v_n}
	\end{equation}
	where the matrices $A_{n,1}$, $B_{n,2}$, $C_{n,3}$ have full row rank. 
	Here, the numbers \ $r_{2,n}$, $r_{1,n}$, $r_{0,n}$, $v_n$ \ are row-sizes of the block rows of $M_n$. Furthermore, these numbers are invariant under left equivalent transformations. Thus, we can call them the \emph{local characteristic invariants of the SiDE \eqref{eq1.2}}.
\end{lem}
\begin{proof} The block diagonal form \eqref{block-upper-M} is obtained directly by consecutively compressing the block columns $A_{n}$, $B_{n}$, $C_{n}$ of $M_n$ via Lemma \ref{lem1.4}. In details, we have that 
	\begin{align*}
	\mbox{ rows of } A_{n,1} &\mbox{ form the basis of the space } \range(A^T_n) ,			\\  
	\mbox{ rows of } B_{n,2} &\mbox{ form the basis of the space } \range(T^T_0(A_n)\,  B_n)^T , \\
	\mbox{ rows of } C_{n,3} &\mbox{ form the basis of the space } \range\left(T^T_0 \left(\m{A_n^T \ B_n^T}^T \right)\,  C_n\right)^T .
	\end{align*}
	Moreover, from \eqref{block-upper-M}, we obtain the following identities
	\bens
	r_{2,n} &=& \rank(A_{n}), \\
	r_{1,n} &=& \rank(\m{ A_{n} & B_{n} }) - \rank(A_{n}), \\
	r_{0,n} &=& \rank(\m{ A_{n} &  B_{n} & C_{n} }) - \rank(\m{A_{n} & B_{n}}), \\
	v_n     &=& m - r_{2,n} - r_{1,n} - r_{0,n} \ ,
	\eens
	which proves the second claim.	
\end{proof}

Analogous to the continuous-time case, we will apply an \emph{algebraic approach} (see \cite{Bru09,MehS06}), which aims to reformulate \eqref{eq1.2} into a so-called \emph{strangeness-free} form, as stated in the following definition.

\begin{definition}\label{Def strangeness-free} 
\rm{
	(\cite{LinNT16})
	System  \eqref{eq1.2} is called \emph{strangeness-free} if there exists a pointwise nonsingular matrix sequence $\{P_n\}_{n\geq n_0}$ such that by scaling the SiDE  \eqref{eq1.2} at each point $n$ with $P_n$, then we obtain a new system of the form
	\be\label{def sfree}
	\pm{\hr_2 \\ \hr_1 \\ \hr_0 \\\hv } \quad 
	\m{\hA_{n,1} \\ 0 \\ 0 \\ 0} x(n\!+\!2) + \m{ \hB_{n,1} \\ \hB_{n,2} \\ 0 \\ 0} x(n\!+\!1) + \m{ \hC_{n,1} \\ \hC_{n,2} \\ \hC_{n,3} \\ 0} x(n)  \!=\! \m{ \hat{f}_1(n) \\  \hat{f}_2(n) \\  \hat{f}_3(n) \\ \hat{f}_4(n)}, 
	\ee
	for all $n\geq n_0$, where matrix $\m{\hA^T_{n,1} \ \hB^T_{n+1,2} \ \hC^T_{n+2,3}}^T$ always has full row rank. 
}
\end{definition}


In order to perform an algebraic approach, an additional assumption below is usually needed.

\begin{assumption}\label{Ass1} 
\rm{
Assume that the local characteristic invariants $r_{2,n}$, $r_{1,n}$, $r_{0,n}$ become global, i.e., they are constant for all $n\geq n_0$. \ Furthermore, assume that two matrix sequences $\left\{ \m{ A^T_{n,1} & B^T_{n,2} & C^T_{n,3} }^T \right\}_{n\geq n_0}$ and $\left\{ \m{ B^T_{n,2} & C^T_{n,3} }^T \right\}_{n\geq n_0}$ have constant rank for all $n\geq n_0$.
}
\end{assumption}

\begin{rmk} 
\rm{
	Following directly from the proof of Lemma \ref{lem2.1}, we see that Assumption \ref{Ass1} is satisfied if and only if five following constant rank conditions are satisfied
	%
	%
	\begin{align} \label{constant rank assumption} 
	& \rank(A_{n}) \!\equiv\! \const, \ \rank(\m{A_{n} & B_{n}}) \!\equiv\! \const , \ \rank(\m{A_{n} &  B_{n} & C_{n}}) \!\equiv\! \const, \notag \\
	& \rank(T^T_0(A_n) \, B_n) \!\equiv\! \const, \ \rank \left(T^T_0 \left(\m{A_n^T \ B_n^T}^T \right) \, C_n\right) \!\equiv\! \const
	\end{align}
}
\end{rmk}

\begin{rmk}
\rm{
In \eqref{def sfree}, the quantities $\hr_2$, $\hr_1$, and $\hr_0$ are dimensions of the second order dynamics part, the first order dynamics part, and the algebraic (zero order) part, respectively. Furthermore, $r_2 + r_1$ is exactly the degree of freedoms. 
}
\end{rmk}

Let us call the number \ $r_u := 3 r_2 + 2 r_1 + r_0$ the \emph{upper rank} of system \eqref{eq1.2}. Clearly, $r_u$ is invariant under left equivalence transformations. Rewrite \eqref{eq2.2} block row-wise, we obtain the following system for all $n\geq n_0$.
\begin{subequations}\label{eq2.5}
	\begin{alignat}{3}
	\label{eq2.5a} A_{n,1} x(n+2)  + B_{n,1} x(n+1)  + C_{n,1} x(n) &= f_1(n),& \quad r_2 \ \mbox{equations}, \\
	\label{eq2.5b} B_{n,2} x(n+1)  + C_{n,2} x(n) &= f_2(n), & \quad r_1 \ \mbox{equations}, \\
	\label{eq2.5c} C_{n,3} x(n) &= f_3(n), & \quad r_0 \ \mbox{equations}, \\
	\label{eq2.5d} 0 &= f_4(n),& \quad v \ \mbox{equations}.
	\end{alignat}
\end{subequations}
Since the matrices $A_{n,1}$, $B_{n,2}$, $C_{n,3}$ have full row rank, the number of scalar difference equations of order $2$ (resp. $1$, and $0$) in  \eqref{eq1.2} is exactly $r_{2}$ (resp. $r_{1}$ and $r_{0}$), while $v$ is the number of redundant equations. \ 
Now we are able to define the shift-forward operator $\De$, which acts on some or whole equations of system \eqref{eq2.5}. This operator maps each equation of system \eqref{eq2.5} 
at the time instant $n$ to the equation itself at the time $n+1$, for example
\be\label{2.5c-shift} 
\De: C_{n,3} x(n) = f_3(n) \mapsto C_{n+1,3} x(n+1) = f_3(n+1).
\ee
Clearly, under Assumption \ref{Ass1}, this shift operator can be applied to equations of system \eqref{eq2.5}. 
In order to reveal all hidden constraints of \eqref{eq2.5} we propose the idea that for each $j=1, 2$, we use equations of order less than $j$ to reduce the number of scalar equations of order $j$. 
This task will be performed as follows. Firstly, by applying Lemma \ref{lem1.2} to two matrix pairs $(B_{n,2},C_{n+1,3})$ and $\left( A_{n,1}, \sm{B_{n+1,2} \\ C_{n+2,3}} \right)$, we obtain matrix sequences 
$\{ S^{(i)}_{n} \}_{n\geq n_0}$, $i=1, 2$, and $\{ Z^{(j)}_{n} \}_{n\geq n_0}$, $j=1,...,5$, of appropriate sizes such that for all $n\geq n_0$, the following conditions hold true.
\begin{enumerate}
	\item[i)] For $i=1, 2$, the matrices $\sm{S^{(i)}_{n} \\ Z^{(i)}_{n}} \in \r^{r_i,r_i}$ are orthogonal.
	\item[ii)] The following identities hold true.
	\begin{subequations}\label{eq2.10}
		\begin{alignat}{3}
		\label{eq2.10a} Z^{(1)}_{n} B_{n,2} + Z^{(3)}_{n} C_{n+1,3} \ &=& \ 0, \\
		\label{eq2.10b} Z^{(2)}_{n} A_{n,1} + Z^{(4)}_{n} B_{n+1,2} + Z^{(5)}_{n} C_{n+2,3} \ & = & \ 0.
		\end{alignat}	
	\end{subequations}	
	\item[iii)] Both matrix pairs $\left(S^{(2)}_{n} A_{n,1}, \sm{S^{(1)}_{n}B_{n+1,2} \\ C_{n+2,3}} \right)$, $\left(S^{(1)}_{n} B_{n,2},C_{n+1,3}\right)$ have no hidden redundancy. 
\end{enumerate}  

Now we will transform the SiDE \eqref{eq1.2} as in Lemma \ref{lem2.10} below.

\begin{lem}\label{lem2.10} Assume that Assumption \ref{Ass1} is satisfied. Let the matrix sequences $\{ S^{(i)}_{n} \}_{n\geq n_0}$, $i=1, 2$, and $\{ Z^{(j)}_{n} \}_{n\geq n_0}$, $j=1,...,5$ be defined as above. 
	Then the SiDE  \eqref{eq1.2} has exactly the same solution set as the transformed system 	
	\[
	\pm{ d_{2} \\ s_2 \\ \hline \\[-0.35cm] d_1 \\ s_1 \\ \hline \\[-0.35cm] r_0 \\ v} \ 
	\m{S^{(2)}_{n} A_{n,1}  & S^{(2)}_{n} B_{n,1}      			   			& \ S^{(2)}_{n} C_{n,1} 	\\
		0				& Z^{(2)}_{n} B_{n,1} + Z^{(4)}_{n} C_{n+1,2}   & \ Z^{(2)}_{n} C_{n,1}  \\ \hline \\[-0.35cm]
		0 				& S^{(1)}_{n} B_{n,2}  	   			   			& \ S^{(1)}_{n} C_{n,2}  \\
		0 				& 0                			   					& \ Z^{(1)}_{n} C_{n,2} 	\\ \hline \\[-0.35cm]
		0				& 0				 			   					& \ C_{n,3}     \\
		0 		    	& 0              			   					& \ 0 	} 
	\m{x(n+2) \\ x(n+1) \\ x(n)} =  
	\]
	\begin{equation}\label{eq2.11}
	= \m{ S^{(2)}_{n} f_1(n) \\ 
		Z^{(2)}_{n} f_1(n) + Z^{(4)}_{n} f_2(n+1) + Z^{(5)}_{n} f_3(n+2) \\ \hline \\[-0.35cm] 
		S^{(1)}_{n} f_2(n) \\ 
		Z^{(1)}_{n} f_2(n) + Z^{(3)}_{n} f_3(n+1) \\ \hline 
		f_3(n) \\  
		f_4(n)}\ \mbox{ for all } n\geq n_0.
	\end{equation}
	Furthermore, both matrix pairs $\left(S^{(2)}_{n} A_{n,1}, \sm{S^{(1)}_{n} B_{n+1,2} \\ C_{n+2,3}} \right)$, $\left(S^{(1)}_{n} B_{n,2},C_{n+1,3}\right)$ have no hidden redundancy. 
\end{lem}
\begin{proof} 
	The proof is not too difficult but rather lengthy and technical, so we leave it to \ref{App0}.
\end{proof}
Consider system \eqref{eq2.11}, we see that the upper rank of the behavior matrix is 
\bens 
r_u^{new} &\leq& 3 d_2 + 2 (s_2 + d_1) + (s_1 + r_0) \\ 
&=& 3 (r_2-s_2) + 2 (s_2 + r_1 - s_1) + (s_1 + r_0) \\ 
&=& r-(s_2+s_1) \leq r.
\eens
In conclusion, after performing a so-called \emph{index reduction step}, which passes from \eqref{eq2.5} to \eqref{eq2.11}, we have reduced the upper rank $r_u$ at least by $s_2+s_1$. 
Continue in this fashion until $s_1=s_2=0$, we obtain the following algorithm.
\begin{algorithm}[H]
	\caption{Index reduction steps for SiDEs at the time point $n$}
	\label{Alg1}
	\textbf{Input:} The SiDE  \eqref{eq1.2} and its behavior form \eqref{eq2.2}. \\
	\textbf{Output:} A strangeness-free SiDE of the form \eqref{def sfree} and the strangeness-index $\mu$.
	\begin{algorithmic}[1]
		\State	Set $i=0$. 
		%
		%
		\State  Transform the behavior matrix $\m{A_n & B_n & C_n}$ to the block upper triangular form
		\[\tM_n :=
		\m{
			A_{n,1} & B_{n,1}  & C_{n,1}     \\
			0       & B_{n,2}  & C_{n,2}      \\
			0       & 0        & C_{n,3}  \\ 
			0       & 0        & 0
		},
		\]
		where all the matrices $A_{n,1}$, $B_{n,2}$, $C_{n,3}$ on the main diagonal have full row rank.
		The system now takes the form \eqref{eq2.5}.
		\IF{both matrix pairs $\left( A_{n,1}, \sm{B_{n+1,2} \\ C_{n+2,3}}\right)$ and $\left(B_{n,2},C_{n+1,3} \right)$ have no hidden redundancy} set  $\mu=i$ and STOP. 
		\ELSE{ set $i := i + 1$} 
		\State Find the matrices $S^{(j)}_{n}$, $j=1, 2$, and $Z^{(j)}_{n}$, $j=1,...,5$.
		%
		%
		%
		%
		%
		\State Transform the system to the new form \eqref{eq2.11} as in Lemma \ref{lem2.10}. 				
		\ENDIF
		\State Go back to Step 2 with the updated behavior matrix.		
	\end{algorithmic}
\end{algorithm}

After each index reduction step the upper rank $r^{i}_u$ has been decreased at least by $s^{i}_2+s^{i}_1$, so Algorithm \ref{Alg1} terminates after a finite number $\mu$ of iterations, which will be called the \emph{strangeness-index} of the SiDE  \eqref{eq1.2}.

\begin{thm}\label{thm2} Consider the SiDE \eqref{eq1.2} and assume that Assumption \ref{Ass1} is satisfied for any $n$ and any $i$ considered within the loop, such that the strangeness-index $\mu$ is well-defined by Algorithm \ref{Alg1}. Then the SiDE  \eqref{eq1.2} has the same solution set as the strangeness-free SiDE
	\be\label{sfree-SiDE}
	\pm{r^{\mu}_{2} \\ r^{\mu}_{1} \\ r^{\mu}_{0} \\ v^{\mu}} \qquad
	\m{\hA_{n,1}& \hB_{n,1}    & \hC_{n,1}     \\
		0		& \hB_{n,2}    & \hC_{n,2}     \\
		0		&    0          & \hC_{n,3} \\ 
		0       & 0          & 0}
	\m{x(n+2) \\ x(n+1) \\ x(n)} = \m{ \hg_1(n) \\  \hg_2(n) \\  \hg_3(n) \\ \hg_4(n)}\ \mbox{ for all } \ n\geq n_0, 
	\ee
	where the matrix $\m{\hA^T_{n,1} \ \hB^T_{n+1,2} \ \hC^T_{n+2,3}}^T$ has full row rank for all $n\geq n_0$, and the functions $\hg_2$ and $\hg_3$ consist of the components of $f(n),\,f(n+1),\dots,f(n+2\mu)$ (at most).
\end{thm}
\begin{proof}
	The proof is a direct consequence of Algorithm \ref{Alg1}, where the matrix $\m{\hA^T_{n,1} \ \hB^T_{n+1,2} \ \hC^T_{n+2,3}}^T$ has full row rank due to Lemma \ref{lem1.3}.
\end{proof}

To illustrate Algorithm \ref{Alg1}, we consider the following example.

\begin{example}\label{exp2}
\rm{
Given a parameter $\a \in \r$, we consider the second order SiDE
	\be\label{eq2.13}
	\m{ 1 & n\!+\!1 & n\!\!+\!\!4 \\  0 & 0 & 0 \\ 0 & 0 & 0} \!x(n\!+\!2) \!+\!
	\m{ 0 & \a & 2n\!+\!3 \\ 1 & n & 1 \\  0 & 0 & 0}\! x(n\!+\!1) \!+\!
	\m{ 0 & n\!+\!1 & 0 \\ 0 & 0 & n \\  0 & 0 & n\!+\!1}\! x(n) \!=\! \m{f_{1}(n) \\ f_{2}(n) \\ f_{3}(n)}, 
	\ee
	for all $n\geq 0$. Fortunately, the behavior matrix 
	\begin{equation*}
	M \!=\! 
	\left[
	\begin{tabular}{ccc|ccc|ccc}
	1 & $\ n\!+\!1$   & $n\!+\!4 $	& 0   & $\a$ & $2n+3$ & 0 & $n+1 $	& 0 \\  \hline 
	0 & 0 			& 0 		& 1 \ & $n$  & 1 	 & 0 & 0	 	& $n$ \\ \hline 
	0 & 0 			& 0 		& 0 \ & 0  & 0 	 & 0 & 0	 	& $n+1$
	\end{tabular} 
	\right]
	\!=\! \m{	A_{n,1} & B_{n,1}  & C_{n,1}     \\
		0       & B_{n,2}  & C_{n,2}      \\
		0       & 0        & C_{n,3} }
	\end{equation*}
	is already in the block diagonal form, so we do not need to perform Step 2 in Algorithm \ref{Alg1}. Furthermore, all constant rank conditions required in Assumption \ref{Ass1} are satisfied. 
	We observe that 
	\begin{align*}
	B_{n+1,2} &= \m{1 \ & \ n+1 & \ 1}, \quad C_{n+1,2} = \m{0 \ & \ 0	& \ n+1}, \\
	C_{n+1,3} &= \m{0 \ & \ 0	& \ n+2}, \quad C_{n+2,3} = \m{0 \ & \ 0	& \ n+3}. 
	\end{align*}
	By directly verifying, we see that the matrix pair $\left(A_{n,1}, \sm{B_{n+1,2} \\ C_{n+2,3} } \right)$ has hidden redundancy, while the pair $(B_{n,2},C_{n+1,3})$ does not. 
	Now we choose $S^{(2)}_{n} = [ \ ]$, $Z^{(2)}_{n}=1$, $Z^{(4)}_{n}=-1$, $Z^{(5)}_{n}=-1$. 
	Notice that the fact $Z^{(5)}_{n}$ is non-empty leads to the appearance of $f_3(n+2)$. Furthermore, the resulting system \eqref{eq2.11} reads
	\begin{equation}\label{eq2.14}
	\m{ 0  & \a & n\!+\!2 \\ 1 & n & 1 \\ \hline  0 & 0 & 0} \! x(n\!+\!1) \!+\! \m{ 0 & n\!+\!1 & \! 0 \\ 0 & 0 & \! n \\\hline  0 & 0 & \! n\!+\!1} \! x(n) \!=\! \m{f_{1}(n) \!-\! f_2(n\!+\!1) \!-\! f_3(n\!+\!2) \\ f_{2}(n) \\\hline f_{3}(n)}. 
	\end{equation}
	Since the leading coefficient matrix associated with $x(n+2)$ becomes zero, so for notational convenience we do not write this term. 
	Go back to Step 3, we see that the following two cases may happen.\\
	i) If $\alpha \not= 0$, then Algorithm \ref{Alg1} terminates here, and the strangeness-index is $\mu=1$. The number of time-shift appear in the inhomogeneity $f$ in the strangeness-free formulation \eqref{eq2.14} is $2$.\\
	ii) If $\a = 0$, then the matrix pair $\left(\sm{0 & \a & n\!+\!2 \\ 1 \ & n & 1},  \m{0 & 0	& n+2} \right)$ have hidden redundancy. Now we choose $S^{(1)}_{n} = \m{1 & 0}$, $Z^{(1)}_{n}=\m{0 & 1}$, $Z^{(3)}_{n}=-\m{0 & 1}$. The resulting system \eqref{eq2.11} now reads
	\begin{align}\label{eq2.15}
	& \m{ 1 \ & n & 1 \\ \hline  0 \ & 0 & 0 \\ 0 & 0 & 0} x(n+1) + \m{ 0 & 0 & n \\ \hline  0 & n+1 & 0 \\ 0 & 0 & n+1} x(n) \notag \\
	& = \m{f_{2}(n) \\ \hline f_{1}(n) - f_2(n+1) - f_3(n+2) - f_3(n+1) \\ f_{3}(n)}. 
	\end{align}
	Algorithm \ref{Alg1} terminates here, and the strangeness-index is $\mu=2$. However, the number of time-shifts appearing in the inhomogeneity $f$ in the strangeness-free formulation \eqref{eq2.15} remains $2$.
}
\end{example}

As a direct consequence of Theorem \ref{thm2}, we obtain the solvability for \eqref{eq1.2} as follows.

\begin{corollary}\label{Cor1} Under the assumption of Theorem \ref{thm2}, the following statements hold true.
	\begin{enumerate}
		\item[i)] The corresponding IVP for the SiDE \eqref{eq1.2} is solvable if and only if either $v^{\mu}=0$ or $\hg_4(n)\!=\!0$ for all $n\geq n_0$. Furthermore, it is uniquely solvable if, in addition, we have $d = m-v^{\mu}$.
		\item[ii)] The initial condition \eqref{eq1.3} is consistent if and only if the following equalities hold.
		\bens 
		\hB_{n_0,2}x_1 + \hC_{n_0,2} x_0 &=& \hg_2(n_0), \\
		\hC_{n_0,3}x_0 &=& \hg_3(n_0).
		\eens 
	\end{enumerate}	
\end{corollary}

Another direct consequence of Theorem \ref{pro3.1} is that we can obtain an inherent regular difference equation as follows.
\begin{corollary}\label{Cor2} Assume that the IVP \eqref{eq1.2}-\eqref{eq1.3} is uniquely solvable for any consistent initial condition.
	Under the assumption of Theorem \ref{thm2}, the solution $x$ to this IVP is also a solution to the \emph{(implicit) inherent regular difference equation}
	\be\label{underlying}
	\m{\hA_{n,1} \\ \hB_{n+1,2} \\ \hC_{n+2,3}} x(n+2) + \m{ \hB_{n,1}  \\ \hC_{n+1,2} \\ 0} x(n+1) + \m{\hC_{n,1} \\ 0 \\ 0}	x(n) = \m{ \hg_1(n) \\  \hg_2(n+1) \\  \hg_3(n+2)},
	\ee
	where the matrix $\m{\hA^T_{n,1} & \hB^T_{n+1,2} & \hC^T_{n+2,3}}$ is invertible for all $n\geq n_0$.	
\end{corollary}

\begin{rmk}
\rm{	Unlike the procedures in \cite{Bru09,LosM08,MehS06}, we do not change the variable $x$. This approach permits us to simplify significantly the condensed forms in these references. We emphasize that as in \eqref{constant rank assumption}, we only require five constant rank conditions within one step of index reduction, instead of seven as in \cite{MehS06}. Therefore, by this way 
	the domain of application for SiDEs (and also for DAEs in the continuous-time case) will be enlarged.
	This approach is also useful for the control analysis of the descriptor system \eqref{eq1.1}, as will be seen in the next section. 
}
\end{rmk}

\begin{rmk}
\rm{	
	i) Within one loop of Algorithm \ref{Alg1}, for each $n$, we have used four Singular Value Decompositions (SVDs) to remove the hidden redundancies in two matrix pairs. The total cost depends on the problems itself, i. e., depending on sizes of the matrix pairs which applied SVDs. Nevertheless, it does not exceed ${\cal O}(m^2 d^2)$.\\
	ii) Unfortunately, since $Z_n^{(3)}$, $Z_n^{(4)}$,$Z_n^{(5)}$ are not orthogonal, in general Algorithm \ref{Alg1} could not be stably implemented. For the numerical solution to the IVP \eqref{eq1.2}-\eqref{eq1.3}, we will consider a suitable numerical scheme in Section \ref{Sec4}.
}
\end{rmk}

\section{Regularization of second order descriptor systems}\label{Sec3}

Based on the index reduction procedure for SiDEs in Section \ref{Sec2}, in this section we construct the strangeness-index concept for the descriptor system \eqref{eq1.1}.
The solvability analysis for first order descriptor systems with variable coefficients have been carefully discussed in \cite{ByeKM97,KunMR01,Rat97}. Nevertheless, for second order descriptor systems, this problem has been rarely considered. We refer the interested readers to \cite{LosM08,Wun08} for continuous-time systems. 

It is well known that in regularization procedures of continuous-time systems, one should avoid differentiating equations that involve an input function, due to the fact that it may not be differentiable. We will also keep this spirit, and hence, will not shift any equation that involve an input function, since it may destroy the causality of the considered system, as in Example \ref{Exa1}. 
Instead of it, we will also incorporate proportional state and first order feedback within each index reduction step of the regularization procedure, as will be seen later. \ Now let us present two auxiliary lemmas, which will be
very useful later.
\begin{lem}\label{lem1.5} Given four matrices $\chA$, $\chB$, $\chC$ in $\r^{m,d}$ and $\chD$ in $\r^{m,p}$. Let us consider the following matrices whose columns span orthogonal bases of the associated vector spaces
	\[
	\begin{array}{lllll}
	T_{1} 	& \mbox{basis of } \kernel(\chA^T), 							& \mbox{ and } & T_{1,\perp} & \mbox{basis of } \range(\chA),			\\
	W_{1} 	& \mbox{basis of } \kernel(T_1^T \chD)^T,   					& \mbox{ and } & W_{1,\perp} & \mbox{basis of } \range(T_1^T \chD),		\\
	&  																& 			   &  \Jd  	 & := W_{1,\perp}^T  T_{1}^T  \chD,       	\\       	       
	\Jbone  & := W_1^T T_1^T \chB, 											& \mbox{ and } &  \Jbtwo  	 & := W_{1,\perp}^T T_{1,\perp}^T \chB,   	\\
	\Jcone  & :=  W_{1}^T T_{1}^T \chC, 									& \mbox{ and } &  \Jctwo  	 & := W_{1,\perp}^T T_{1}^T \chC,           \\       
	T_{2} 	& \mbox{basis of } \kernel( \Jbone^T), 							& \mbox{ and } & T_{2,\perp} & \mbox{basis of } \range( \Jbone),			\\
	T_{3} 	& \mbox{basis of } \kernel( \Jbtwo^T),							& \mbox{ and } & T_{3,\perp} & \mbox{basis of } \range( \Jbtwo), \\
	T_{4} 	& \mbox{basis of } \kernel(T_{2}^T  \Jcone)^T, 					& \mbox{ and } & T_{4,\perp} & \mbox{basis of } \range(T_{2}^T  \Jcone). 		
	\end{array}
	\]
	Then the following assertions hold true.
	\begin{enumerate}
		\item[i)] The matrices $\m{T_{i,\perp}^T \ T_{i}^T }^T$, $i=1,...,4$, and $\m{W_{1,\perp}^T \ W_1^T}^T$ are orthogonal.
		\item[ii)]	The matrices $T_{1,\perp}^T \chA$, $T_{2,\perp}^T  \Jbone$,  $T_{3,\perp}^T  \Jbtwo$, $T^T_{4,\perp} T_{2}^T  \Jcone$, and $ \Jd$ have full row rank.
		\item[iii)] Moreover, there exists an orthogonal matrix $\chU$ such that 
		\be\label{eq1.6}
		\chU \m{\chA & \chB & \chC & \vline & \chD}
		\!=\! 
		\m{\chA_1 & \chB_1 & \chC_1 & \vline & \chD_1 \\
			0     & \chB_2 & \chC_2 & \vline & 0	 \\
			0     & 0	   & \chC_3 & \vline & 0	 \\
			0     & 0	   & 0		& \vline & 0	 \\ \hline \\[-.35cm]
			0     & \chB_4 & \chC_4 & \vline & \chD_4 \\
			0     & 0	   & \chC_5 & \vline & \chD_5 		},
		\ee
		where the matrices $\chA_1$, $\chB_2$, $\chB_4$, $\chC_3$, $\m{\chD_4^T \ \chD_5^T}^T$ have full row rank.
	\end{enumerate}	
\end{lem}
\begin{pf} The first two claims followed directly from Lemma \ref{lem1.4}. To prove the third claim, we construct the desired matrix $\chU$
	as follows
	\[
	\chU :=
	\m{  I &\vline  &  	&		 & \vline & 	\\ \hline \\[-0.35cm]
		&\vline  & I  & 			 & \vline & 	\\
		&\vline  &    & T_{4,\perp}^T & \vline &     \\ 
		&\vline  & 	& T_{4}^T 		 & \vline &     \\ \hline \\[-0.35cm]
		&\vline  & 	&			 & \vline &  I}  
	\cdot
	\m{I &\vline &  			 & \vline & 	\\ \hline \\[-0.35cm]
		&\vline & T_{2,\perp}^T & \vline &     \\ 
		&\vline & T_{2}^T 		 & \vline &     \\ \hline \\[-0.35cm]
		&\vline & 				 & \vline &  T_{3,\perp}^T   \\ 
		&\vline & 		 		 & \vline &  T_{3}^T } 
	\cdot 
	\m{I &  \vline & \\ \hline \\[-0.35cm] & \vline & W_{1}^T \\ & \vline & W_{1,\perp}^T} 
	\cdot
	\m{T_{1,\perp}^T \\ T_1^T} \ .
	\]
	Thus, we have that 
	\[
	\chU \m{\chA & \chB & \chC & \vline & \chD} \!=\! 
	\m{
		T_{1,\perp}^T \chA 	 & \ T_{1,\perp}^T \chB  		& \ T_{1,\perp}^T \chC  		& \vline & T_{1,\perp}^T \chD \\[.1cm] 
		0 					 & \ T_{2,\perp}^T  \Jbone 	 	& \ T_{2,\perp}^T  \Jcone 	 	& \vline & 0 \\[.1cm]
		0 					 & 0 	 						& \  T_{4,\perp}^T T_{2}^T  \Jcone 	 	& \vline & 0 \\[.1cm]
		0 					 & 0 	 						& 0												 	 	& \vline & 0 \\ \hline \\[-0.35cm]
		0 					 & \ T_{3,\perp}^T  \Jbtwo  & \ T_{3,\perp}^T  \Jctwo 	& \vline & T_{3,\perp}^T  \Jd \\[.1cm]
		0 					 & 0  							& \ T_{3}^T  \Jctwo 	& \vline & T_{3}^T  \Jd} \ .
	\]
	Due to the parts i) and ii), we see that this is exactly the desired form \eqref{eq1.6}.	
\end{pf}

\begin{lem}\label{lem1.6} Let $P\in\r^{p,d}$, $Q \in \r^{q,d}$ be two full row rank matrices, where $p+q\leq d$. Then the following assertions hold true.
	\begin{enumerate}
		\item[i)] There exists a matrix $F\in \r^{d,d}$ such that $H := \sm{P \\ QF}$ has full row rank.
		\item[ii)] For any $G\in \r^{q,d}$, there exists a matrix $F\in \r^{d,d}$ such that $\sm{P \\ G+QF}$ has full row rank.
	\end{enumerate}	
\end{lem}
\begin{pf}
	i) First we consider the SVDs of $P$ and $Q$ that reads
	\begin{equation*}
	U_P P V_P = \m{\Si_P & \ 0_{p,d-p}}, \quad U_Q Q V_Q = \m{\Si_Q & \ 0_{q,d-q}},
	\end{equation*}
	where $\Si_P$, $\Si_Q$ are nonsingular, diagonal matrices, and $0_{p,d-p}$ (resp. $0_{q,d-q}$) are the zero matrix of size $p$ by $d-p$ (resp. $q$ by $d-q$).\\
	By choosing $F:= V_Q \ \sm{0 & I_q \\ I_{d-q} & 0} \ V_P^{T}$ we see that
	\[ 
	\m{U_P & 0 \\ 0 & U_Q} \ \m{P \\ QF} \ V_P = \m{U_P P V_P \\ U_Q QF V_P} = \m{\Si_P & 0_{p,d-p-q} & 0_{p,q} \\ 0_{q,p} & 0_{p,d-p-q} & \Si_Q}, \]
	and hence, the claim i) is proven.\\
	ii) Clearly, in case that the matrix $F$ is very big in norm, then $G$ is only a small perturbation, and hence for sufficiently large $\eta$, by choosing 
	\[ F := \eta V_Q \ \m{0 & I_q \\ I_{d-q} & 0} \ V_P^{T} \ , \]%
	we obtain the full row rank property of $\sm{P \\ G+QF}$. 
\end{pf}

\begin{rmk}\label{rem1.1}
	\rm{
		It should be noted that, the proof of Lemmas \ref{lem1.5} and \ref{lem1.6} are constructive, and all the matrices $T_{i,\perp}$, $T_{i}$, $i=1,...,4$, $W_{1,\perp}$, $W_1$ and $F$ can be stably computed.
	}
\end{rmk}

In the following theorem, we give the condensed form for system \eqref{eq1.1}.

\begin{thm}\label{pro5.1}
	i) Consider the descriptor system \eqref{eq1.1}. Then there exist two pointwise nonsingular matrix sequences $\{U_n\}_{n\geq n_0}$, $\{V_n\}_{n\geq n_0}$ such that 
	by scaling \eqref{eq1.1} with $U_n$ and changing  $u(n) = V_n v(n)$, $\tf(n):=U_nf(n)$, we can transform \eqref{eq1.1} to the system
	\begin{equation}\label{eq3.3}
	\pm{r_{2,n} \\ r_{1,n} \\ r_{0,n} \\ \hline \vphi_{1,n} \\ \vphi_{0,n} \\  v_n} \quad
	\m{A_{n,1}  & B_{n,1}    & C_{n,1}     \\
		0    	& B_{n,2}    & C_{n,2}     \\
		0    	& 0          & C_{n,3}     \\  \hline
		0    	& B_{n,4}    & C_{n,4}     \\
		0    	& 0          & C_{n,5}     \\  
		0    	& 0          & 0} 
	\m{x(n\!+\!2) \\ x(n\!+\!1) \\ x(n) } \!+\! 
	\m{ D_{n,1} 	    & 0		    & 0		  \\
		0 	    & 0		    & 0		  \\
		0 	    & 0		    & 0        \\	\hline 
		0      	    & \Si_{n,1} & 0 		 \\ 
		0     	    & 0		    & \Si_{n,0}       \\
		0     	    & 0             & 0 } \underbrace{\m{v_1(n) \\ v_2(n) \\ v_3(n)}}_{v(n)} \!=\! \tf(n)
	\end{equation}
	for all $n\geq n_0$. Here, sizes of the block rows are $r_{2,n}$, $r_{1,n}$, $r_{0,n}$, $\vphi_{1,n}$, $\vphi_{0,n}$, $v_n$, the matrices $A_{n,1}$, $B_{n,2}$, $B_{n,4}$, $C_{n,3}$ are of full row rank and the matrices $\Si_{n,1}$, $\Si_{n,0}$ are nonsingular and diagonal.\\
	ii) Furthermore, if the matrix $\m{A^T_{n,1} \ B^T_{n+1,2} \ C^T_{n+2,3}}^T$ is of full row rank for all $n\geq n_0$ then there exists a first order feedback of the form
	\be\label{eq3.5}
	v(n) = F_{n,1} x(n+1) + F_{n,0} x(n)
	\ee
	such that the closed loop system  
	\[
	A_n x(n+2) + \left(B_{n} + D_{n} F_{n,1} \right) x(n+1) + \left(C_{n} + D_{n} F_{n,0} \right) x(n) = f(n)
	\]
	is strangeness-free.	
\end{thm}
\begin{proof}
i)	First we apply Lemma \ref{lem1.5} to four matrices $A_n$, $B_n$, $C_n$ and $D_n$ to obtain the matrix $U_n$ that satisfies \eqref{eq1.6}. Decompose the matrix $\m{\chD_4^T \ \chD_5^T}^T$ via one SVD, we then obtain the block 
$ \sm{
	0      	    & \Si_{n,1} 	    & 0 		 \\ 
	0     		& 0		    & \Si_{n,0} } \ .
$
Finally, we use Gaussian elimination to nullify all non-zero matrices in the two columns of $\chD$ that contain $\Si_{n,1}$ and $\Si_{n,0}$, and hence, we obtain
\begin{align}\label{eq3.2}	
&	(U_n \m{A_{n} & B_{n} & C_{n}}, \ U_n D_n V_n)  \notag \\ 
&= \left( \m{A_{n,1}  & B_{n,1}    & C_{n,1}     \\
	0    & B_{n,2}    & C_{n,2}     \\
	0    & 0          & C_{n,3}     \\  \hline
	0    & B_{n,4}    & C_{n,4}     \\
	0    & 0          & C_{n,5}     \\  
	0    & 0          & 0}, \
\m{ D_{n,1} 		& 0			& 0		  \\
	0 			& 0		    & 0		  \\
	0 			& 0		    & 0        \\	\hline 
	0      	    & \Si_{n,1} 	    & 0 		 \\ 
	0     		& 0		    & \Si_{n,0}       \\
	0     		& 0         & 0 } \right) \quad \mbox{ for all } n\geq n_0.
\end{align}	
This directly leads us the desired system \eqref{eq3.3}. \\
ii) By applying Lemma \ref{lem1.6} for $P=\m{ A^T_{n,1} & B^T_{n+1,2} & C^T_{n+2,3}}^T$, $Q=\sm{0 & \Si_{n,1} & 0  \\ 0 & 0 & \Si_{n,0} }$ and $G=\sm{B_{n+1,4} \\ C_{n+2,5}}$, we see that there exist two matrix sequences $\{F_{n,1}\}_{n\geq n_0}$, $\{F_{n,0}\}_{n\geq n_0}$ 
such that by choosing the feedback of the form \eqref{eq3.5} then the matrix 
\[
\m{ A_{n,1} \\ B_{n+1,2} \\ C_{n+2,3} \\ \hline \\[-.35cm]  B_{n+1,4} + \m{0 & \Si_{n,1} & 0} F_{n+1,1} \\[.05cm] C_{n+2,5} + \m{0 & 0 & \Si_{n,0}} F_{n+2,0} }
\]
has full row rank for all $n\geq n_0$. This completes the proof.
\end{proof}

In order to build an index reduction procedure for \eqref{eq1.1}, we also need the following assumption.
\begin{assumption}\label{Ass2}
\rm{
Assume that the local characteristic invariants $r_{2,n}$, $r_{1,n}$, $r_{0,n}$, $\vphi_{1,n}$, $\vphi_{0,n}$, $v_n$, become global, i.e., they are constant for all $n\geq n_0$. 
}
\end{assumption}

From Theorem \ref{pro5.1}, we see that we only need to remove the hidden redundancies in the upper part of \eqref{eq3.3} as follows. By performing one index reduction step for the upper part of \eqref{eq3.3}, as in section \ref{Sec2}, we obtain the following lemma.

\begin{lem}\label{lem5.2} Assume that the upper part of the descriptor system \eqref{eq3.3} is not strangeness-free. 
	Then for each input sequence $\{v(n)\}_{n\geq n_0}$, it has exactly the same solution set as the following system
	\begin{equation}\label{eq3.4}
	\pm{\tr_{2} \\ \tr_{1} \\ \tr_{0} \\ \hline \vphi_{1} \\ \vphi_{0} \\  \tv} \quad
	\m{\tA_{n,1}    & \tB_{n,1}    & \tC_{n,1}     \\
		0    		& \tB_{n,2}    & \tC_{n,2}     \\
		0    		& 0            & \tC_{n,3}     \\  \hline \\[-.35cm]
		0    		& B_{n,4}      & C_{n,4}     \\
		0    		& 0            & C_{n,5}     \\  
		0    		& 0            & 0} 
	\m{x(n\!+\!2) \\ x(n\!+\!1) \\ x(n) } \!+\! 
	\m{ \tD_{n,1} 	& 0			& 0		  \\
		0	 		& 0		    & 0		  \\
		0 			& 0		    & 0        \\	\hline \\[-.35cm]
		0      	    & \Si_{n,1} 	    & 0 		 \\ 
		0     		& 0		    & \Si_{n,0}       \\
		0     		& 0         & 0 } \m{v_1(n) \\ v_2(n) \\ v_3(n)} \!=\! \tf(n) 
	\end{equation}
	for all $n\geq n_0$. Here, we have $\tr_2=r_2-s_2$, $\tr_1=r_1+s_2-s_1$, $\tr_0 = r_0+s_1$, $\tv\geq v$ for some $s_2>0$, $s_1>0$. Furthermore, both pairs $\left(\tA_{n,1},\m{\tB_{n,2}^T \ \tC_{n,3}^T}^T \right)$ and $\left(\tB_{n,2},\tC_{n,3} \right)$ have no hidden redundancy.
\end{lem}
\begin{proof}
	System \eqref{eq3.4} is directly obtained by applying Lemma \ref{lem2.10} to the upper part of \eqref{eq3.3}. To keep the brevity of this paper, we will omit the details here.
\end{proof}

Similar to the observation made in section \ref{Sec2}, we also see that an \emph{index reduction step}, which passes system \eqref{eq3.3} to the new form \eqref{eq3.4} 
has reduced the upper rank $r^u$ by at least $s_2+s_1$. Continue in this way until $s_2 = s_1 = 0$, finally we obtain a strangeness-free descriptor system in the next theorem.

\begin{thm}\label{thm4}
	Consider the descriptor system \eqref{eq1.1}. Furthermore, assume that Assumption \ref{Ass2} is fulfilled whenever needed. Then for each fixed input sequence $\{u(n)\}_{n\geq n_0}$, system \eqref{eq1.1} has the same solution set as the so-called \emph{strangeness-free descriptor system} 
	\be\label{sfree-descriptor}
	\pm{\hr_{2} \\ \hr_{1} \\ \hr_{0} \\ \hline \\[-0.35cm] \hat{\vphi}_{1} \\ \hat{\vphi}_{0} \\  \hv} \ 
	\m{\hA_{n,1}& \hB_{n,1}    & \hC_{n,1}     \\
		0		& \hB_{n,2}    & \hC_{n,2}     \\
		0		&    0          & \hC_{n,3}		 \\ \hline \\[-0.35cm]
		0    & \hB_{n,5}    & \hC_{n,5}     \\
		0    & 0          & \hC_{n,6}     \\ 
		0    & 0          & 0} \!
	\m{x(n\!+\!2) \\ x(n\!+\!1) \\ x(n)} \!+\! 
	\m{\hD_{n,1} \\ 0 \\ 0  \\ \hline \\[-0.35cm] \hD_{n,4} \\ \hD_{n,5} \\ 0} \! u(n) 
	\!=\! \m{\hf_1(n) \\ \hf_2(n) \\ \hf_3(n) \\ \hline \\[-0.35cm] \hf_4(n) \\ \hf_5(n) \\ \hf_6(n) }\ \mbox{ for all } n\geq n_0,
	\ee
	where the matrices $\m{\hA^T_{n,1} & \hB^T_{n+1,2} & \hC^T_{n+2,3}}^T$, $\m{\hD^T_{n,4} & \hD^T_{n,5}}^T$ have full row rank for all $n\geq n_0$. 
\end{thm}
\begin{proof}
	By repeating index reduction steps until the upper rank $r^u$ stop decreasing, we obtain the system 
	\begin{equation*}
	\pm{\hr_{2} \\ \hr_{1} \\ \hr_{0} \\ \\[-0.35cm] \hat{\vphi}_{1} \\ \hat{\vphi}_{0} \\  \hv} \quad
	\m{\hA_{n,1}& \hB_{n,1}    & \hC_{n,1}     \\
		0		& \hB_{n,2}    & \hC_{n,2}     \\
		0		&    0          & \hC_{n,3}		 \\ \hline \\[-0.35cm]
		0       & \hB_{n,5}    & \hC_{n,5}     \\
		0    	& 0          & \hC_{n,6}     \\ 
		0    	& 0          & 0}
	\m{x(n\!+\!2) \\ x(n\!+\!1) \\ x(n)} \!+\! 
	\m{\hD_{n,11}		& 0			& 0		  \\
		0 			& 0		    & 0		  \\
		0 			& 0		    & 0        \\	\hline \\[-0.35cm]
		0      	& \hat{\Si}_{n,1} & 0			 \\ 
		0     		& 0		    & \hat{\Si}_{n,0}        \\ 
		0     		& 0         & 0       	
	} v(n) \!=\! \m{\hf_1(n) \\ \hf_2(n) \\ \hf_3(n) \\ \hline \\[-0.35cm] \hf_4(n) \\ \hf_5(n) \\ \hf_6(n) } 
	\end{equation*}
	for all $n\geq n_0$, where the matrix $\m{\hA^T_{n,1} \ \hB^T_{n+1,2} \ \hC^T_{n+2,3}}^T$ has full row rank for all $n\geq n_0$. The new input sequence $\{v(n)\}_{n\geq n_0}$ satisfies $u(n)=V_nv(n)$, where $V_n$ is nonsingular for all $n\geq n_0$. Transform back $v(n) = V^{-1}_n u(n)$ and set
	\[ 
	\m{\hD_{n,1} \\ 0 \\ 0 \\ \hline \\[-0.35cm] \hD_{n,4} \\ \hD_{n,5} \\ 0} :=
	\m{\hD_{n,11}		& 0			& 0		  \\
		0 			& 0		    & 0		  \\
		0 			& 0		    & 0        \\	\hline \\[-0.35cm]
		0      	& \hat{\Si}_{n,1} & 0			 \\ 
		0     		& 0		    & \hat{\Si}_{n,0}        \\ 
		0     		& 0         & 0       	
	} V_n^{-1}, 
	\]
	we obtain exactly the strangeness-free descriptor system \eqref{sfree-descriptor}.
\end{proof}

As a direct corollary of Theorem \ref{thm4}, we obtain the existence and uniqueness of a solution to the closed-loop system via feedback as follows.

\begin{corollary}\label{coro5.1} Under the conditions of Theorem \ref{thm4}, the following statements hold true. \\
	i) There exists a first order feedback of the form \eqref{eq3.5} such that the closed-loop system is solvable if and only if either $\hv=0$ or $\hf_6(n)=0$ for all $n\geq n_0$. \\
	ii) Furthermore, the solution to the corresponding IVP (of the closed-loop system) is unique if and only if in addition, $d=m-\hv$.
\end{corollary}

\begin{rmk}
\rm{
It should be noted that, analogously to SiDEs, each index reduction step of the descriptor system \eqref{eq1.1} also makes use of Lemma \ref{lem2.10}, where the matrices $Z^{(i)}_n$, $i=3,4,5$, may not be orthogonal. Furthermore, in 
Theorem \ref{pro5.1}, two matrices $U_n$, $V_n$ are only nonsingular but not orthogonal. 
Therefore, in general, the strangeness-free formulation \eqref{sfree-descriptor} could not be stably computed. For the numerical treatment of (continuous-time) second order DAEs, in \cite{Wun08} a different approach was developed. We will modify it for SiDEs/descriptor systems in the next section.
}
\end{rmk}

\begin{rmk}
\rm{
Another interesting method in the study of descriptor systems is the \emph{behavior approach}, where we do not distinguish the state $x$ and an input $u$ but combine them in one \emph{behavior vector}. Then \eqref{eq1.1} will become a SiDE of this behavior variable, and hence, we can apply the results in section \ref{Sec2} for this system. 
Nevertheless, due to the reinterpretation of variables, this approach may alter the strangeness-free form \eqref{sfree-descriptor}. To keep the brevity of this research, we will not present the details here. For the interested readers, we refer to \cite{KunMR01,Rat97,Rat97a} for the case of first order DAEs, and to \cite{Wun08} for the case of second order DAEs.
}
\end{rmk}
\section{Difference arrays associated with second order SiDEs/descriptor systems}\label{Sec4}

In two previous sections, to analyze the solvability of the SiDE  \eqref{eq1.2} or of the descriptor system \eqref{eq1.1}, first one needs to bring it into the strangeness-free form. 
Nevertheless, sometime this task is not feasible, for example when Assumptions \ref{Ass1} or \ref{Ass2} is violated at some index reduction steps. These difficulties have also been observed for continuous-time systems of first or higher orders in \cite{KunMR01,Wun08}.
A breakthrough, thanks to Campbell \cite{Cam87} while considering DAEs, is to differentiate a given system a number of times and put everyone of them, including the original one, into a so-called \emph{inflated system}. Then the strangeness-free formulation will be determined by appropriate selection of equations inside this inflated system. In this section we will examine this approach to the descriptor system \eqref{eq1.1}. The analysis for SiDEs of the form \eqref{eq1.2} can be obtained by simply setting $D_n$ to be $0_{m,p}$ for all $n$.
We further assume the following condition.

\begin{assumption}\label{Ass3}
\rm{
	Consider the descriptor system \eqref{eq1.1}. Assume that there exists a first order feedback of the form \eqref{eq3.5} such that the corresponding IVP of the closed-loop system is uniquely solvable.
}	
\end{assumption}

Notice that, in case of the SiDE \eqref{eq1.2}, Assumption \ref{Ass3} means that the IVP \eqref{eq1.2}-\eqref{eq1.3} is uniquely solvable. \ 
Now let us introduce the \emph{difference-inflated system of level $\ell \in \N$} as follows.
\bens
A_{n} x(n\!+ \!2) \!+ \! B_{n} x(n\!+ \!1) \!+ \! C_{n} x(n) \!+ \! D_n u(n) &\!=\!& f(n), \notag \\
A_{n\!+ \!1} x(n\!+ \!3) \!+ \! B_{n\!+ \!1} x(n\!+ \!2) \!+ \! C_{n\!+ \!1} x(n\!+ \!1) \!+ \! D_{n\!+ \!1} u(n\!+ \!1) &\!=\!& f(n\!+ \!1), \notag	\\
&\dots& \\
A_{n\!+\!\ell} x(n\!+\!\ell\!+\!2) \!+\! B_{n\!+\!\ell} x(n\!+\!\ell\!+\!1) \!+\! C_{n\!+\!\ell} x(n\!+\!\ell) \!+\! D_{n\!+\!\ell} u(n\!+\!\ell) &\!=\!& f(n\!+\!\ell) \ . \notag
\eens
We rewrite this system as
\[
\underbrace{\m{
		C_n	& B_n     & A_n     &         &         &  &  &  \\ 
		& C_{n\!+\!1} & B_{n\!+\!1} & A_{n\!+\!1} &         &  &  &  \\ 
		&         & \ddots  & \ddots  & \ddots  &  &  &  \\ 
		&         &         & \ddots  & \ddots  & \ddots &  &  \\ 
		&         &         &         & C_{n\!+\!\ell} & B_{n\!+\!\ell} & A_{n\!+\!\ell}
}}_{=:\cM}
\! 
\underbrace{\m{x(n) \\ x(n\!+\!1)  \\ x(n\!+\!2) \\ \vdots \\ x(n\!+\!\ell)}}_{=:\cX} 
+
\]
\be\label{inflated}
+ \underbrace{\m{
		D_n &         &         & \\
		& D_{n+1} &  		& \\
		&         & \ddots  & \\
		&         &         & D_{n+\ell}
}}_{=:\cN}
\underbrace{\m{u(n) \\ u(n\!+\!1)  \\ \vdots \\ u(n\!+\!\ell)}}_{=:\cU} 
= \! \underbrace{\m{f(n) \\ f(n\!+\!1)  \\ \vdots \\ f(n\!+\!\ell)}}_{=:\cG}\ \mbox{ for all } n\geq n_0.
\ee
\begin{definition}\label{shift index}
\rm{
Suppose that the descriptor system \eqref{eq1.1} satisfies Assumption \ref{Ass3}. Let $\ell$ be the minimum number such that a strangeness-free descriptor system of the form \eqref{sfree-descriptor} can be extracted from \eqref{inflated} 
by using elementary matrix-row operations. Then the so-called \emph{shift-index} of \eqref{eq1.1}, denoted by $\nu$, is set by $\ell/2$ if $\ell$ is even and by $(\ell+1)/2$ otherwise.
}
\end{definition}
%

We give the relation between this shift-index $\nu$ and the strangeness-index $\mu$ in the following proposition.

\begin{proposition}\label{pro3.1}
	Suppose that the descriptor system \eqref{eq1.1} satisfies Assumption \ref{Ass3}. If the strangeness-index $\mu$ is well-defined, then so is the shift-index $\nu$. Furthermore, we have that $\nu \leq \mu$. 
\end{proposition}
\begin{proof} The claim is straight forward, since every reformulation step performed in Algorithm \ref{Alg1} is a consequence of an inflated system \eqref{inflated} with $\ell = 2\mu$ or $2\mu - 1$. 
\end{proof}

\begin{rmk}
\rm{
As will be seen later in Example \ref{exa3}, for second order SiDEs, the shift-index can be strictly smaller than the strangeness index.
}
\end{rmk}

\begin{rmk}
\rm{
Restricted to the case of first order SiDEs (i.e., $A_n=0$ for all $n\geq n_0$), the strangeness-index $\mu$ defined in this paper is equal to the forward strangeness-index proposed by Br\"ull \cite{Bru09}. For second order system, our strangeness-index is analogous to the one for continuous-time systems proposed by Mehrmann and Shi \cite{MehS06}, and by Wunderlich \cite{Wun08}. We, however, emphasize that the canonical forms constructed in this research is simpler and more convenient from the theoretical viewpoint. Besides that, similar to the case of continuous-time systems, the strangeness index $\mu$ only gives an upper bound for the number of shift-forward operator that have been used, in order to achieve the strangeness-free form \eqref{sfree-SiDE}. For further details, see Remark 17, \cite{MehS06}. 
}
\end{rmk}

In the following theorem we will answer the question how to derive the strangeness-free formulation \eqref{sfree-descriptor} from \eqref{inflated}. 

\begin{thm}\label{thm6} Assume that the shift index $\nu$ of the descriptor system \eqref{eq1.1} is well-defined. Furthermore, suppose that \eqref{eq1.1} satisfies Assumption \ref{Ass3}. 
	Then any solution to the descriptor system \eqref{eq1.1} is also a solution to the following system
	\begin{equation}\label{sfree form}
	\pm{\hr_{2} \\ \hr_{1} \\ \hr_{0} \\ \hline \\[-0.35cm] \hat{\vphi}_{1} \\ \hat{\vphi}_{0} } \ 
	\m{\hA_{n,1}& \hB_{n,1}    	& \hC_{n,1}     \\
		0		& \hB_{n,2}    	& \hC_{n,2}     \\
		0		&    0          & \hC_{n,3}		 \\ \hline \\[-0.35cm]
		0    	& \hB_{n,5}    	& \hC_{n,5}     \\
		0    	& 0          	& \hC_{n,6}    } \!
	\m{x(n\!+\!2) \\ x(n\!+\!1) \\ x(n)} \!+\! 
	\m{\hD_{n,1} \\ 0 \\ 0  \\ \hline \\[-0.35cm] \hD_{n,4} \\ \hD_{n,5}} \! u(n) 
	\!=\! \m{\hG_{n,1} \\ \hG_{n,2} \\ \hG_{n,3} \\ \hline \\[-0.35cm] \hG_{n,4} \\ \hG_{n,5} }\ \mbox{ for all } n\geq n_0,
	\end{equation}
	where the matrices $\m{\hA^T_{n,1} & \hB^T_{n+1,2} & \hC^T_{n+2,3}}^T$, $\m{\hD^T_{n,4} & \hD^T_{n,5}}^T$ have full row rank for all $n\geq n_0$.
	Furthermore, we have that \ $\sum_{i=0}^{2}\hr_i + \sum_{i=0}^{1}\hat{\vphi}_i = d$, or equivalently, 
	\be\label{eq4.12}
	\rank \left(\m{\hA_{n,1} \\ \hB_{n+1,2} \\ \hC^T_{n+2,3}} \right) + \rank \left(\m{\hD_{n,4} \\ \hD_{n,5}} \right) = d \ .
	\ee
\end{thm}
\begin{proof}
Assume that $\nu$ is already known, we now construct an algorithm to select the strangeness-free descriptor system \eqref{sfree-descriptor} from the inflated system \eqref{inflated}. 
For notational convenience, we will follow the MATLAB language, \cite{matlab}. Consider the following spaces and matrices
\be\label{eq4.2}
\begin{array}{ll}
	\cW &:= \m{ \cM(:,3n+1 : end) & \quad \cN(:,n+1 : end) }, \\
	U_1 & \mbox{ basis of } \kernel(\cW^T), \ \mbox{and} \ U_{1,\perp} \mbox{ basis of } \range(\cW).	
\end{array}
\ee
Due to Lemma \ref{lem1.4}, we have that $U_1^T \cW = 0$ and $U_{1,\perp}^T \cW$ has full row rank. Furthermore, the matrix $\m{U_1 \ U_{1,\perp}}^T$ is nonsingular, and hence, system \eqref{inflated} is equivalent to the coupled system below. 
\begin{alignat}{3}
\label{eq4.3a} & U_1^T \cM(:,1 : 3n) \m{x(n) \\ x(n+1) \\ x(n+2)} + U_1^T \cN(:,1:n) u(n) = U_1^T \cG, \\
\label{eq4.3b} & U^T_{1,\perp}  \cW \m{x(n\!+\!3) \\ \vdots \\ x(n\!+\!\nu) \\ \hline  u(n\!+\!1) \\ \vdots \\ u(n\!+\!\nu)} 
+ U^T_{1,\perp} \m{ \cM(:,1\!:\!3n) & \cN(:,1\!:\!n)} \! \m{x(n) \\ x(n\!+\!1) \\ x(n\!+\!2) \\ \hline u(n)} \!=\! U_{1,\perp}^T \! \cG.
\end{alignat}
Due to the full row rank property of $U_{1,\perp}^T \cW$, we see that \eqref{eq4.3b} plays no role in the determination of the strangeness-free descriptor system \eqref{sfree-descriptor}. Thus, \eqref{sfree-descriptor} is a consequence of \eqref{eq4.3a}. 
For notational convenience, let us rewrite system \eqref{eq4.3a} as	
\[
\m{\chA & \ \chB & \ \chC & \vline \ \chD} \m{x(n\!+\!2) \\ x(n\!+\!1) \\ x(n) \\ \hline u(n)} =  \chG. 
\]
Scaling this system with the matrix $\chU$ obtained in Lemma \ref{lem1.5}, we have	
\begin{align}\label{eq4.7}
\m{	\chA_1 	 			 & \ \chB_1  		& \ \chC_1  	& \vline & \chD_1 \\
	0 					 & \ \chB_2 	 	& \ \chC_2 	 	& \vline & 0 \\
	0 					 & 0 	 			& \  \chC_3 	& \vline & 0 \\ 
	0 					 & 0 	 			& 0				& \vline & 0 \\ \hline \\[-0.35cm]
	0 					 & \ \chB_4  		& \ \chC_4 		& \vline & \chD_4 \\
	0 					 & 0  				& \ \chC_5 		& \vline & \chD_5		} 
\! \m{x(n\!+\!2) \\ x(n\!+\!1) \\ x(n) \\ \hline u(n)} 
=  \m{\chG_1 \\ \chG_2 \\ \chG_3 \\ 0 \\ \hline \\[-0.35cm]  \chG_4 \\ \chG_5 } \ ,
\end{align}
where the matrices $\chA_1$, $\chB_2$,  $\chB_4$, $\chC_3$, and $\sm{ \chD_4^T \  \chD_5^T}^T$ have full row rank. 
Notice that the presence of the $0$ block on the right hand side vector is due to the existence of a solution (Assumption \ref{Ass3}).
Applying Lemma \ref{lem1.4} consecutively for two following matrix pairs $\left(\chB_2, \chC_3 \right)$, $\left(\chA_1, \sm{\chB_2^T \ \chC_3^T}^T \right)$, we obtain two orthogonal matrices $\sm{S^{(i)}_{n} \\ Z^{(i)}_{n}} \in \r^{r_i,r_i}$, $i=1,2$ such that both pairs $\left(S^{(1)}_{n} \chB_2, \chC_3 \right)$, $\left(S^{(2)}_{n}\chA_1, \sm{\chB_2^T \ \chC_3^T }^T \right)$ have no hidden redundancy. Scaling the first and second block row equations of \eqref{eq4.7} with $S^{(2)}_{n}$ and $S^{(1)}_{n}$ respectively, we obtain 
	\begin{equation*}
	\m{	S^{(2)}_{n} \chA_1 	 & \ S^{(2)}_{n}\chB_1  		& \ S^{(2)}_{n} \chC_1  	& \vline & \ S^{(2)}_{n} \chD_1 \\
		0 					 & \ S^{(1)}_{n} \chB_2 	 	& \ S^{(1)}_{n} \chC_2 		& \vline & 0 		
	} 
	\! \m{x(n\!+\!2) \\ x(n\!+\!1) \\ x(n) \\ \hline u(n)} = \m{S^{(2)}_{n}  \chG_1 \\ S^{(1)}_{n} \chG_2} \ .
	\end{equation*}
	Combining these equations with the third, fifth and sixth block equations of \eqref{eq4.7}, we obtain the system
	\begin{equation}\label{eq4.9}
	\m{	S^{(2)}_{n} \chA_1 	 & \ S^{(2)}_{n}\chB_1  		& \ S^{(2)}_{n} \chC_1  	& \vline & \ S^{(2)}_{n} \chD_1 \\
		0 					 & \ S^{(1)}_{n} \chB_2 	 	& \ S^{(1)}_{n} \chC_2 		& \vline & 0 					\\
		0 					 & 0 	 						& \ \chC_3					& \vline & 0 \\ \hline \\[-0.35cm]
		0 					 & \ \chB_4  					& \ \chC_4 					& \vline & \chD_4 \\
		0 					 & 0  							& \ \chC_5 					& \vline & \chD_5		} 
	\! \m{x(n\!+\!2) \\ x(n\!+\!1) \\ x(n) \\ \hline u(n)} 
	=  \m{S^{(2)}_{n} \chG_1 \\ S^{(1)}_{n} \chG_2 \\ \chG_3 \\ \hline \\[-0.35cm]  \chG_4 \\ \chG_5 } \ .
	\end{equation}
	which is exactly our desired system \eqref{sfree form}.	Moreover, due to Lemma \ref{lem1.3}, the matrix $\m{(S^{(2)}_{n} \chA_1)^T  & (S^{(1)}_{n} \chB_2)^T & \chC^T_3}^T$ 
	has full row rank, and the identity \eqref{eq4.12} holds true due to Assumption \ref{Ass3}. \\

Finally, we will prove that system \eqref{sfree form} is not affected by left equivalence transformation. Let us assume that  \eqref{eq1.1} is left equivalent to the SiDE 
		\begin{equation}\label{eq4.5}
		\tA_{n} x(n+2) + \tB_{n} x(n+1) + \tC_{n} x(n)  + \tD_{n} u(n) = \tf(n) \ \mbox{ for all } n\geq n_0.
		\end{equation}
		Thus, there exists a pointwise nonsingular matrix sequence $\{P_n\}_{n\geq n_0}$ such that 
		\[ \m{\tA_{n} \ \tB_{n} \ \tC_{n} \ \tD_{n}} = P_n \m{A_{n} \ B_{n} \ C_{n} \ D_{n} } \mbox{ and } \tf(n) = P_n f(n)\ \mbox{ for all } n\geq n_0. \]
		Therefore, the difference-inflated system of level $\ell$ for system \eqref{eq4.5} takes the form
		\begin{equation}\label{eq4.6}
		\tcM \cX + \tilde{\cN} \cU = \tilde{\cG}, 
		\end{equation}
		where the matrix coefficients are
		\[
		\tcM \!=\! \diag(P_n,...,P_{n+\ell}) \ \cM, \ \tilde{\cN} \!=\! \diag(P_n,...,P_{n+\ell}) \ \cN, \ \tilde{\cG} \!=\! \diag(P_n,...,P_{n+\ell}) \cG.
		\]
		This follows that two systems \eqref{inflated} and \eqref{eq4.6} are left equivalent, which finishes the proof.
\end{proof}
We summarize our result in the following algorithm.
\begin{algorithm}[H]
	\caption{Strangeness-free formulation for SiDEs using difference arrays}
	\label{Alg2}
	\textbf{Input:} The SiDE  \eqref{eq1.1}.\\
	\textbf{Output:} The strangeness-free descriptor system \eqref{sfree form} and the minimal number of shifts $\ell$.
	\begin{algorithmic}[1]
		\State Set $\ell := 0$.
		\State Construct the difference-inflated system of level $\ell$, and rewrite it in the form \eqref{inflated}.
		\State Find $U_1$ as in \eqref{eq4.2} and construct system \eqref{eq4.3a}.
		\State Find $\chU$ as in Lemma \ref{lem1.5} and construct system \eqref{eq4.7}.
		\State Find the matrices $S^{(1)}_{n}$, $S^{(2)}_{n}$ in the process used to remove the hidden redundancies in two matrix pairs $\left(\chB_2, \ \chC_3 \right)$, $\left(\chA_1, \m{\chB^T_2 \ \chC^T_3 }^T \right)$, respectively.
		\State Construct the system \eqref{eq4.9}. 
		\IF{$\rank \m{\hA^T_{n,1} \ \hB^T_{n+1,2} \ \hC^T_{n+2,3}}^T + \rank \m{\hD^T_{n,4} \ \hD^T_{n,5}}^T = d $} STOP. 
		\ELSE{ set $\ell := \ell + 1$ and go to 2} 
		\ENDIF
	\end{algorithmic}
\end{algorithm}

In order to illustrate Algorithm \ref{Alg2}, we consider the following two examples. 

\begin{example}\label{exa3}
\rm{
Let us revisit system \eqref{eq2.13} for the case $\a=0$. In this system, $D_n=0$ for all $n\geq 0$. 
For $\ell = 2$, the inflated system \eqref{inflated} reads
	\be\label{eq4.10}
	\left[
		\begin{array}{rrr|rrr}
			C_n	& B_n     		& A_n     	     & 0        	   	& 0        		 		  		\\ 
			0	& C_{n\!+\!1}	& B_{n\!+\!1}    & A_{n\!+\!1}   	& 0        		   			\\ 
			0	& 0        	    & C_{n\!+\!2}	 & \undermat{=:\cW}{ B_{n\!+\!2} & \quad A_{n\!+\!2}} 			 	
		\end{array}
	\right]
	\! 
	\m{x(n) \\ x(n\!+\!1)  \\ x(n\!+\!2) \\ \hline x(n\!+\!3) \\ x(n\!+\!4)}
	= \m{f(n) \\ f(n\!+\!1)  \\ f(n\!+\!2)}
	\ee
	Let	$U_1$ be the basis of $\kernel(\cW^T)$. We then determine system \eqref{eq4.3a} by scaling \eqref{eq4.10} with $U_1^T$. The resulting system reads
	\be\label{eq4.11}
	U_1^T  \m{	C_n	& B_n     		& A_n     	      		\\ 
		0	& C_{n\!+\!1}	& B_{n\!+\!1}   				\\ 
		0	& 0        	    & C_{n\!+\!2}	
	}
	\! 
	\m{x(n) \\ x(n\!+\!1)  \\ x(n\!+\!2)}
	= U_1^T \m{f(n) \\ f(n\!+\!1)  \\ f(n\!+\!2)} \ .
	\ee 
	Finally, by performing Steps 6 to 10 we can extract the strangeness-free form \eqref{eq2.15} from \eqref{eq4.11}. Thus, we conclude that the shift index is $\nu=1$, which is the same as the shift index in the case $\alpha\neq 0$. We recall Example \ref{exp2}, in which it is shown that the strangeness indices in the two cases are different.
}
\end{example}

\begin{example}\label{exa4}
\rm{
	A singular system of second order differential equations, which describes a three link robot arm \cite{Hou94a}, is given by
	\begin{equation*}
	\m{M_0 & 0 \\ 0 & 0} \ddot{x}(t) + \m{G_0 & 0 \\ 0 & 0} \dot{x}(t) + \m{K_0 & H_0^T \\ H_0 & 0} x(t) = \m{B_0 \\ 0} u(t),
	\end{equation*}
	where $M_0$ represents the nonsingular mass matrix, $G_0$ the coefficient matrix associated
	with damping, centrifugal, gravity, and Coriolis forces, $K_0$ the stiffness matrix, and $H_0$ the constraint. A simple discretized version of this system with the stepsize $h$ takes the form
	\begin{align*}
	& \m{M_0 & 0 \\ 0 & 0} \dfrac{x(n+2)-2x(n+1)+x(n)}{h^2}  + \m{G_0 & 0 \\ 0 & 0} \dfrac{x(n+2)-x(n)}{2h} \\ & + \m{K_0 & H_0^T \\ H_0 & 0} x(n+1) = \m{B_0 \\ 0} u(n+1).
	\end{align*}
	As a simple example, let us take $M_0 = G_0 = K_0 = H_0= B_0 = 1$, $h=0.01$. Then Algorithm \ref{Alg2} terminates after two steps and hence, the shift index is $\nu=1$ for all $n\geq n_0$. 
	Furthermore, we notice that no matter central, forward or backward difference is chosen for discretizing the derivative $\dot{x}(t)$, the shift index remains unchanged $\nu=1$. Of course, the resulting strangeness-free descriptor systems are different.
}
\end{example}

\section{Conclusion}
By using the algebraic approach, we have analyzed the solvability of second order SiDEs/descriptor systems, 
based on the derived condensed forms constructed under certain constant rank assumptions. In comparison to the previously known procedures \cite{MehS06,Wun06}, we have reduced the number of constant rank conditions in every index reduction step from seven to five. This would enlarge the domain of application for SiDEs (and also for DAEs). However, requiring constant rank assumptions in the discrete-time case seems less nature than in the continuous-time case. To overcome this limitation, we also consider the difference-array method, which is numerically stable, to obtain the strangness-free form. The index theory together with the two algorithms presented in this paper can be extended without difficulty to arbitrarily high order SiDEs/descriptor systems.
We also notice that the backward time case ($n\leq n_0$) can be directly extended from the forward time case, as it has been done in \cite{Bru09}. The analysis of the two-way case, which happens while considering boundary value problems for SiDEs, 
is under our on-going research. Furthermore, the condensed forms presented in this work also motivate further study on the staircase form for second order systems, which would be an interesting extension of the classical result for first order systems, e.g. \cite{Son13}.

\vspace{0.5cm}
\noindent
{\bf Acknowledgment} The authors would like to thank the anonymous referee for very helpful comments and suggestions that led to improvements of this paper. 


\bibliographystyle{abbrv}

\begin{thebibliography}{10}
%
\bibitem{Aga00}
	R.~Agarwal.
	\newblock {\em Difference Equations and Inequalities: Theory, Methods, and
		Applications}.
	\newblock Chapman \& Hall/CRC Pure and Applied Mathematics. CRC Press, 2000.
	
	\bibitem{Bru09}
	T.~Br\"ull.
	\newblock Existence and uniqueness of solutions of linear variable coefficient
	discrete-time descriptor systems.
	\newblock {\em Linear Algebra Appl.}, 431(1-2):247--265, 2009.
	
	\bibitem{ByeKM97}
	R.~Byers, P.~Kunkel, and V.~Mehrmann.
	\newblock Regularization of linear descriptor systems with variable
	coefficients.
	\newblock {\em {SIAM} J. Cont.}, 35:117--133, 1997.
	
	\bibitem{Cam87}
	S.~L. Campbell.
	\newblock Comment on controlling generalized state-space (descriptor) systems.
	\newblock {\em Internat. J. Control}, 46:2229--2230, 1987.
	
	\bibitem{Ela13}
	S.~Elaydi.
	\newblock {\em An Introduction to Difference Equations}.
	\newblock Undergraduate Texts in Mathematics. Springer New York, 2013.
	
	\bibitem{GolV96}
	G.~H. Golub and C.~F. {Van Loan}.
	\newblock {\em Matrix Computations}.
	\newblock The Johns Hopkins University Press, Baltimore, MD, 3rd edition, 1996.
	
	\bibitem{HaM12}
	P.~Ha and V.~Mehrmann.
	\newblock Analysis and reformulation of linear delay differential-algebraic
	equations.
	\newblock {\em Electr. J. Lin. Alg.}, 23:703--730, 2012.

	\bibitem{HaMS14}
	P.~Ha, V.~Mehrmann, and A.~Steinbrecher.
	\newblock Analysis of linear variable coefficient delay differential-algebraic
	  equations.
	\newblock {\em J. Dynam. Differential Equations}, 26:889--914, 2014.
	
	\bibitem{Hou94a}
	M.~Hou.
	\newblock A three--link planar manipulator model.
	\newblock Sicherheitstechnische Regelungs- und Me{\ss}technik, Bergische
	Universit{\"a}t--GH Wuppertal, Germany, May 1994.
	
	\bibitem{Kel01}
	W.~Kelley and A.~Peterson.
	\newblock {\em Difference Equations: An Introduction with Applications}.
	\newblock Harcourt/Academic Press, 2001.
	
	\bibitem{KunMR01}
	P.~Kunkel, V.~Mehrmann, and W.~Rath.
	\newblock Analysis and numerical solution of control problems in descriptor
	form.
	\newblock {\em Math. Control, Signals, Sys.}, 14:29--61, 2001.
	
	\bibitem{LinNT16}
	V.H.~Linh, N.T.T.~Nga, and D.D.~Thuan.
	\newblock Exponential stability and robust stability for linear time-varying
	singular systems of second order difference equations.
	\newblock {\em SIAM J. Matr. Anal. Appl.},
	39(1):204--233, 2018.
	
	\bibitem{LosM08}
	P.~Losse and V.~Mehrmann.
	\newblock Controllability and observability of second order descriptor systems.
	\newblock {\em {SIAM} J. Cont. Optim.}, 47(3):1351--1379, 2008.
	
	\bibitem{Lue79}
	D.~Luenberger.
	\newblock {\em Introduction to dynamic systems: theory, models, and
		applications}.
	\newblock Wiley, 1979.
	
	\bibitem{matlab}
	The MathWorks, Inc., Natick, MA.
	\newblock {\em {MATLAB} Version 8.3.0.532 (R2014a)}, 2014.
	
	\bibitem{Meh13}
	V.~Mehrmann.
	\newblock Index concepts for differential-algebraic equations.
	\newblock {\em \sc Encyclopedia Applied Mathematics}, 2014.
	
	\bibitem{MehS06}
	V.~Mehrmann and C.~Shi.
	\newblock Transformation of high order linear differential-algebraic systems to
	first order.
	\newblock {\em Numer. Alg.}, 42:281--307, 2006.
	
	\bibitem{MehT15}
	V.~Mehrmann and D.D.~Thuan.
	\newblock Stability analysis of implicit difference equations under restricted
	perturbations.
	\newblock {\em SIAM J. Matr. Anal. Appl.},
	36(1):178--202, 2015.
	
	\bibitem{Rat97}
	W.~Rath.
	\newblock Derivative and proportional state feedback for linear descriptor
	systems with variable coefficients.
	\newblock {\em Lin. Alg. Appl.}, 260:273--310, 1997.
	
	\bibitem{Rat97a}
	W.~Rath.
	\newblock {\em Feedback Design and Regularization for Linear Descriptor Systems
		with Variable Coefficients}.
	\newblock Dissertation, TU Chemnitz, Chemnitz, Germany, 1997.
	
	\bibitem{Son13}
	E.~Sontag.
	\newblock {\em Mathematical Control Theory: Deterministic Finite Dimensional
		Systems}.
	\newblock Texts in Applied Mathematics. Springer New York, 2013.
	
	\bibitem{Wun06}
	L.~Wunderlich.
	\newblock Numerical treatment of second order differential-algebraic systems.
	\newblock In {\em Proc. Appl. Math. and Mech. (GAMM 2006, Berlin, March 27-31,
		2006)}, volume 6 (1), pages 775--776, 2006.
	
	\bibitem{Wun08}
	L.~Wunderlich.
	\newblock {\em Analysis and Numerical Solution of Structured and Switched
		Differential-Algebraic Systems}.
	\newblock Dissertation, Institut f{\"u}r Mathematik, TU Berlin, Berlin,
	Germany, 2008.
%
\end{thebibliography}


\appendix

\section{Proof of Lemma \ref{lem2.10}}\label{App0}
\begin{pf}
In order to prove this lemma, we will make use of the shifted equation \eqref{2.5c-shift} if the matrix pair $(B_{n,2},C_{n+1,3})$ has hidden redundancy. 
Analogously, if the pair $\left( A_{n,1}, \m{B_{n+1,2}^T \ C_{n+2,3}^T}^T \right)$ has hidden redundancy then we will make use of the shifted equation
\begin{equation}\label{2.5b-shift}
B_{n+1,2} x(n+2)  + C_{n+1,2} x(n+1) = f_2(n+1),
\end{equation}
and may be also the double shifted equation
\be\label{2.5c-doubleshift} 
C_{n+2,3} x(n+2) = f_3(n+2).
\ee
Now we observe that \eqref{eq1.2} has the same solution set as that of the following extended system
\begin{equation}\label{eq2.9}
\pm{ r_{2} \\ r_1 \\ r_0 \\ v \\ \hline  r_0 \\ r_1 \\ r_0 } \
\m{A_{n,1}		& B_{n,1}      			   & C_{n,1} 	\\
	0 			& B_{n,2}  	   			   & C_{n,2}  	\\
	0			& 0			 			   & C_{n,3}     \\
	0 		    & 0            			   & 0 \\ \hline 
	0			& C_{n+1,3}	 			   & 0     \\
	B_{n+1,2}  	& C_{n+1,2}  			   & 0	\\
	C_{n+2,3}		& 0						   & 0     \\
} 
\m{x(n+2) \\ x(n+1) \\ x(n)}  
= \m{ f_1(n) \\ f_2(n) \\ f_3(n) \\  f_4(n) \\ \hline f_3(n+1) \\ f_2(n+1) \\  f_3(n+2) }, 
\end{equation}
for all $n\geq n_0$. Therefore, it suffices to prove that any solution to \eqref{eq2.9} is also a solution to \eqref{eq2.11} and vice versa. \\
\textbf{Necessity:} The main idea is to apply (only) two elementary row transformations below to system \eqref{eq2.9} to obtain \eqref{eq2.11}. 
	\begin{enumerate}
		\item[i)] scaling a block row equation with a nonsingular matrix,
		\item[ii)] adding to one row a linear combination of some other rows.
	\end{enumerate}
By scaling the first (resp., second) block row equation of \eqref{eq2.9} with an orthogonal matrix $\sm{S^{(2)}_{n} \\ Z^{(2)}_{n}}$ (resp., $\sm{S^{(1)}_{n} \\ Z^{(1)}_{n}}$), we obtain an equivalent system to \eqref{eq2.5}, as follows
	\begin{equation}\label{eqA1}
	\pm{ d_{2} \\ s_2 \\ \hline \\[-0.35cm] d_1 \\ s_1 \\ \hline \\[-0.35cm] r_0 \\ v \\ \hline \\[-0.35cm]  r_0 \\ r_1 \\ r_0 } \
	\m{
		S^{(2)}_{n} A_{n,1} & S^{(2)}_{n} B_{n,1}    & S^{(2)}_{n} C_{n,1} 	\\
		Z^{(2)}_{n} A_{n,1} & Z^{(2)}_{n} B_{n,1}    & Z^{(2)}_{n} C_{n,1}  \\ \hline \\[-0.35cm]
		0 			& S^{(1)}_{n} B_{n,2}  	 & S^{(1)}_{n} C_{n,2}  \\
		0 			& Z^{(1)}_{n} B_{n,2}    & Z^{(1)}_{n} C_{n,2} 	\\ \hline \\[-0.35cm]
		0			& 0				 & C_{n,3}     \\
		0 		    & 0              & 0 			\\	\hline 
		0			& C_{n+1,3}					    & 0		     \\
		B_{n+1,2}  	& C_{n+1,2}  			   		& 0	\\
		C_{n+2,3}	& 0				 			    & 0     		
	} \m{x(n+2) \\ x(n+1) \\ x(n)} \!=\! 
	\m{ S^{(2)}_{n} f_1(n) \\ Z^{(2)}_{n} f_1(n) \\ \hline \\[-0.35cm] S^{(1)}_{n} f_2(n) \\ Z^{(1)}_{n} f_2(n) \\ \hline \\[-0.35cm] f_3(n) \\ f_4(n) \\ \hline f_3(n+1) \\  f_2(n+1) \\ f_3(n+2)		}.
	\end{equation}

	By adding the seventh row scaled with $Z^{(3)}_{n}$ to the fourth row of \eqref{eqA1} and making use of \eqref{eq2.10a} we obtain the first hidden constraint
	\begin{equation*}
	Z^{(1)}_{n} C_{n,2} x(n) = Z^{(1)}_{n} f_2(n) + Z^{(3)}_{n} f_3(n+1),
	\end{equation*}
	which is exactly the fourth row of \eqref{eq2.11}.
	
	We continue by adding the seventh row scaled with $Z^{(4)}_{n}$ and the eighth row scaled with $Z^{(5)}_{n}$ to the second row of \eqref{eqA1} and making use of \eqref{eq2.10b} to obtain
	\begin{align*}
	&& \left( Z^{(2)}_{n} B_{n,1} + Z^{(4)}_{n} C_{n+1,2} \right) x(n+1) + Z^{(2)}_{n} C_{n,1} x(n) \notag \\ 
	&& = Z^{(2)}_{n} f_1(n) + Z^{(4)}_{n} f_2(n+1) + Z^{(5)}_{n} f_3(n+2).
	\end{align*}
	This is exactly the second row of \eqref{eq2.11}. Therefore, any solution to \eqref{eq2.5} is also a solution to \eqref{eq2.11}. \\
	\textbf{Sufficiency:} Let $x$ be an arbitrary solution to \eqref{eq2.11}. Thus, $x$ is also a solution to the shifted system
	\[
	\pm{ d_{2} \\ s_2 \\ \hline \\[-0.35cm] d_1 \\ s_1 \\ \hline \\[-0.35cm] r_0 \\ v \\ \hline \\[-0.35cm]  r_0 \\ r_0 } \ 
	\m{S^{(2)}_{n} A_{n,1} & S^{(2)}_{n} B_{n,1}      			   & S^{(2)}_{n} C_{n,1} 	\\
		0			& Z^{(2)}_{n} B_{n,1} + Z^{(4)}_{n} C_{n+1,2}  & Z^{(2)}_{n} C_{n,1}  \\ \hline \\[-0.35cm]
		0 			& S^{(1)}_{n} B_{n,2}  	   			   & S^{(1)}_{n} C_{n,2}  \\
		0 			& 0                			   & Z^{(1)}_{n} C_{n,2} 	\\ \hline \\[-0.35cm]
		0			& 0				 			   & C_{n,3}     \\
		0 		    & 0              			   & 0 		\\	\hline 
		0			& C_{n+1,3}					    & 0		     \\
		C_{n+2,3}	& 0				 			    & 0     	
	} 
	\m{x(n+2) \\ x(n+1) \\ x(n)} = 
	\]
	\begin{align}\label{eqA2}
	&= \m{ S^{(2)}_{n} f_1(n) \\ Z^{(2)}_{n} f_1(n) + Z^{(4)}_{n} f_2(n+1) + Z^{(5)}_{n} f_3(n+2) \\ \hline \\[-0.35cm] S^{(1)}_{n} f_2(n) \\ Z^{(1)}_{n} f_2(n) + Z^{(3)}_{n} f_3(n+1) \\ \hline f_3(n) \\  f_4(n) 
		\\ \hline f_3(n+1) \\  f_3(n+2)	}\ \mbox{ for all } n\geq n_0.
	\end{align}
	Since elementary matrix-row operations are reversible, we can reverse the transformations performed in the necessity part. Consequently, we see that any solution to \eqref{eqA2} is also a solution to \eqref{eqA1}, and hence, this completes the proof.
\end{pf}

\end{document}